\documentclass[11pt,reqno]{amsart}

\usepackage{amssymb,amsmath,amscd,dsfont}
\usepackage[noend]{algorithmic}
\usepackage{algorithm}
\usepackage[mathlines,pagewise,left]{lineno}
\usepackage[dvipsnames]{xcolor}
\usepackage[latin5]{inputenc}
\usepackage{color,soul}
\usepackage{longtable}
\usepackage{calligra}
\usepackage{booktabs}
\usepackage{multirow}

\usepackage{bcol}

\newtheorem{proposition}{Proposition}

\usepackage{eurosym}
\usepackage{longtable}
\usepackage{graphicx}
\usepackage{caption}
\usepackage[labelformat=simple,labelfont=normalfont]{subcaption}

\usepackage{amsmath,amssymb,latexsym,float,epsfig}
\usepackage{epstopdf}
\usepackage{hyperref}
\usepackage[normalem]{ulem}
\usepackage{url}
\usepackage{doi}

\newcommand{\REV}[1]{#1}

\ignore{
\setlength{\oddsidemargin}{0.1in}
\setlength{\evensidemargin}{0.1in}
\setlength{\textwidth}{6.5in}
\setlength{\textheight}{8.0in}
}

\usepackage[colorinlistoftodos,prependcaption,textsize=small]{todonotes}

\begin{document}

%\keywords{airline rescheduling, cruise time controllability, aircraft swapping, mixed integer conic quadratic programming, McCormick inequalities}

\title{accommodating new flights \\ into an existing airline flight schedule}

\author{\"{O}zge \c{S}afak \\ Department of Industrial Engineering, Bilkent University, 06800 Ankara, Turkey, ozge.safak@bilkent.edu.tr \\  \\ Alper Atamt{\"u}rk \\ Industrial Engineering \& Operations Research, University of California, Berkeley, California 94720, USA, atamturk@berkeley.edu \\ \\ M. Selim Akt{\"u}rk \\  Department of Industrial Engineering, Bilkent University, 06800 Ankara, Turkey, akturk@bilkent.edu.tr }

\maketitle

\begin{abstract}
\ignore{We present two novel approaches for airline rescheduling to respond to increasing passenger demand. In both approaches, we alter an existing flight
schedule to accommodate new flights while maximizing the airline's profit.} We present two novel approaches to alter a flight network for introducing new flights while maximizing airline's profit. A key feature of the first approach is to adjust the aircraft cruise speed
to compensate for the block times of the new flights, trading off flying time and fuel burn.
In the second approach, we introduce aircraft swapping as an additional mechanism to
provide a greater flexibility in reducing the incremental fuel cost and adjusting the capacity.
The nonlinear fuel-burn function and the binary aircraft swap and assignment decisions complicate the optimization problem significantly.
We propose strong mixed-integer conic quadratic %(MICQ)
formulations to overcome the computational difficulties. The reformulations enable solving instances with 300 flights from a major U.S. airline optimally within reasonable compute times. \\


\end{abstract}

\begin{center}
Jun 2018, Nov 2018, Apr 2019
\end{center}

\BCOLReport{18.02}

\pagestyle{plain}

\section{Introduction}

\ignore{The U.S. Federal Aviation Administration (FAA) Aerospace Forecast 2017--2037 \cite{FAAreport} reports that the demand for air travel in 2016 grew at the fastest pace since 2005 despite the modest economic growth in the U.S.}

\ignore{For reasons such as tourism, new job opportunities or even a natural disaster, to respond to increasing number of passengers, an airline needs to make immediate changes on its existing flight schedule to introduce new flights in a particular day. For such a short time period, leasing an aircraft just to serve these new flights may not be a feasible option. Therefore,}
While operating a daily flight schedule, several events, such as tourism, business conferences or even a natural disaster, might necessitate introducing new  flights on a particular day. In a relatively short time period, leasing or buying an aircraft just to serve these new flights may not be a feasible option. Therefore,  an airline aims to accommodate new flights with minimum disruption on the existing schedule. In near real time, to accommodate these
new flights into an existing flight schedule, the airline can only make some operational changes: either
it can use the idle times, if any, in the existing schedule, or it can change the departure times of the existing flights, or it can increase the aircraft cruise speed to shorten the flight times. The airline can utilize any one of these alternatives or any combination of them to open up enough time to accommodate the new flights. Increasing the cruise speed, however, has an adverse effect on the fuel burn, which in turn increases the fuel and carbon emission costs, i.e., the most significant component of an airline's operational costs. Since aircraft types have different fuel efficiencies, changing the aircraft assignments
%\todo{I wouldn't bring the fleet assignment problem here as the focus is to modify an existing schedule without major a overhaul.}
may be beneficial in reducing the fuel burn, and consequently decreasing the operational costs.

In this paper, we propose two approaches to accommodate new flights into an existing schedule. The
 first approach carefully adjusts flight departure times and as well as aircraft cruise speed to  allow the required time for operating the new flights. %  for a certain time period.
 Increasing the cruise speed of a flight directly reduces its block time, and thereby opening up space  to accommodate new flights into the flight schedule. Although cruise time reduction provides a great opportunity to add new flights, increasing the speed of an aircraft comes with significant additional cost of fuel burn and CO$_2$ emission. To keep the cost of fuel manageable, \ignore{the decision of fleet type assignment to new flight may be critical. A solution can be improved through assigning} the new flights may be assigned to a fuel-efficient, but smaller aircraft.
  %so that the fuel burn is reduced.
  However, such an assignment may spill some of the passengers due to the insufficient seat capacity, resulting in a loss of revenue. Therefore,
  in order to address this trade-off,
  we propose a second approach, which incorporates an explicit aircraft swapping mechanism together with cruise time controllability. Aircraft swapping provides a greater opportunity in reducing the fuel burn and capturing passenger demand of new flights. Through flight timing and assignment decisions, we trade-off the incremental fuel cost associated with the cruise time compression with the revenue from the passengers. Although the second approach may provide substantial improvements in the airline's profit over the first one, the additional binary swapping decisions and the nonlinear fuel burn function make the optimization problem significantly more difficult to solve. Our aim is to provide a set of alternative schedules with increasing profit at a cost of additional compute time. A decision maker can interactively specify her preferences (or restrictions) and analyze their effect on the airline's profit when introducing new flights into an existing schedule.

Aircraft swapping is a practical way to adjust the capacity based on the demand changes during the booking period. Sherali et al. \cite{Sherali2005} develop a demand-driven re-fleeting model that dynamically re-assigns the aircraft in response to improved passenger demand forecast. They only allow aircraft re-assignment within the same aircraft family to keep the crew assignments unchanged. Jarrah et al. \cite{Jarrah} also re-assign fleet types by limiting the number of changes on the original fleet assignment. Wang and Regan \cite{Wang} examine a dynamic yield management problem when the assigned capacities are subject to a swap. The recent studies show that re-assignment of aircraft to reflect the changing demand yields substantial savings.

In addition to adjusting the capacity, most airlines make use of swap opportunities to build robust aircraft routings or reduce delays in the recovery plans. Ageeva \cite{Ageeva} adds a reward for each opportunity to swap aircraft in an aircraft routing model and encourage overlapping routes to have more swap opportunities in the case of an operational disruption. If two aircraft routings meet at more than one airport, aircraft can be swapped, and then returned to their original routings at a next meeting point. Therefore, if a flight is delayed, swapping the aircraft provides robustness by allowing a flight with high demand to be flown. %Therefore, if a flight is delayed, an overlap, i.e., two \todo{why two?} more meetings of same aircraft, provides a greater flexibility in replacing the aircraft with an alternative.
Akt{\"u}rk et al. \cite{akturk} use the idea of swapping aircraft between flights to reduce the effect of a disruption on the schedule. They provide approximately 30\% cost savings compared to the delay propagation recovery approach. More recently, Arikan et al.\cite{Arikan} use both flight re-timing and aircraft swapping approaches to find minimum cost passenger and aircraft recovery plans. Based on an investigation of over 240,000 domestic routings of 13 major U.S. airlines, Lonzius and Lange \cite{Lonzius} confirm the delay-reducing effect of swap opportunities.

Re-timing approach has been also used to minimize the delay propagation in the entire airline flight network. Chiraphadnakul and Barnhart \cite{virot} and Dunbar \cite{Dunbar} adjust flight departure times to provide slacks across the connections so as to minimize delay propagation. Lan et al. \cite{lans} implement a re-timing approach to minimize misconnected passengers. A novel model to increase the robustness of aircraft routing is presented in Aloulou et al. \cite{aloulou}. They judiciously distribute slacks to connections where they are most needed. Ahmed et al. \cite{ahmed} also adjust flight departure times in their robust weekly aircraft routing problem. Cadarso and de Celis \cite{Cadar} propose a two-stage stochastic programming formulation which updates base schedules in terms of timetable and fleet assignments while considering stochastic demand, and proposes robust itineraries in order to ameliorate miss-connected passengers. Although the aircraft swapping and re-timing approaches have been widely used in the literature,  the novelty in this paper lies in the fact that they are implemented to allow for introducing new flights into the schedule. \ignore{In this paper, we aim to use the aircraft swapping  to adjust the capacity to respond to new scheduling requests as well as to reduce the fuel burn and open up enough space to accommodate the new flights
into an  existing flight schedule.}

In the airline industry there is a realization that cruise speed selections have a significant impact on the airline's profit. Sherali et al. \cite{Sherali2006} state that airline optimization models are quite sensitive to fuel burn. Cook et al. \cite{Cook09} discuss the option of flying faster to ensure the minimum time requirement for the connections of passengers and flying slower for conservation of fuel.

In recent years, the aircraft speed control has been the subject of several research studies such as air traffic management (ATM), airline disruption management, aircraft recovery, and robust schedule design.  The joint work of FAA/Eurocontrol \cite{FAAEurocontrol} emphasizes the importance of speed control for ATM to manage the fuel burn and terminal congestion. While the aim of the proposed methodology in \cite {FAAEurocontrol} is to save fuel by reducing cruise speed, once congestion in a terminal is determined, it is suggested to increase the speed of aircraft at the beginning of a rush period to avoid creating congestion and reduce the overall delay and fuel burn.  Kang and Hansen \cite{Kang} emphasize the importance of accurate flight fuel burn prediction to reduce airline's cost. They showed how ensemble learning techniques can be used to improve flight trip fuel burn prediction. In their study, a novel discretionary fuel estimation approach is proposed to assist dispatchers with better discretionary fuel loading decisions. Kohl et al. \cite{Kohl} discuss the ability to reduce passenger delay costs by accelerating the aircraft in their overview of airline disruption management processes. Marla et al. \cite{Marla} integrate disruption management with flight planning, which enables changes in the flight speed. Using a time-space network, they make multiple copies of flights representing different discrete departure times and cruise speeds. However, in the context of airline operations, this representation leads to a large number of copies of flights to be evaluated in the model. Arikan et al. \cite{Arikan} and Akt{\"u}rk et al. \cite{akturk} express cruise speed as a continuous variable and find an optimal trade-off between increased fuel cost and disruption costs such as delay and spilled passengers costs. To manage disruptions in a less costly manner, airlines are also interested in building robust schedules. More recently, Duran et al. \cite{Duran} and \c{S}afak et al. \cite{Safak} consider the fuel burn and CO$_2$ emission costs associated with the aircraft cruise speed adjustments to ensure the passenger connections with desired probabilities. G{\"u}rkan et. al \cite{Gurkan} also include aircraft cruise speed decisions in an integrated airline scheduling, aircraft fleeting and routing problem.  Different than the existing studies, our key feature is to include cruise time controllability decisions to \ignore{make enough time space} open-up enough time to accommodate new flights into the existing flight schedule. The major difficulty of including controllable cruise time decisions in the model is the nonlinearity of the fuel burn and carbon emission cost functions.  We handle these nonlinearities using formulation strengthening techniques in Akt{\"u}rk et al. \cite{akturk2009strong}. See Atamt\"urk and Narayanan \cite{AN:cmir} for more on strengthening conic mixed-integer programs.

The main contributions of the current paper are as follows:
\begin{itemize}
\item{We propose a new problem of accommodating new flights into an existing flight schedule of a particular day. In this context, for the first time, we introduce the options of
flight re-timing, aircraft cruise speed control and aircraft swapping to open
up enough time for new flights in the schedule.}
\item{We propose strong mixed-integer conic quadratic (MICQ) formulations
to overcome the computational difficulties of nonlinear fuel burn and emission functions as well as the penalty functions of arrival tardiness.}
\item{We improve and strengthen the MICQ formulations by adding Mc-Cormick inequalities. The new formulation with the McCormick inequalities
enables the solution of test instances with 300 flights from a major
U.S. airline optimally within reasonable compute times.}
\end{itemize}

The remainder of the paper is organized as follows. In Section~\ref{Sec:2}, we briefly describe the framework of the problem and then present a numerical example illustrating the benefits of cruise time controllability and the proposed aircraft swapping mechanism. Section~\ref{Sec:3} introduces the mixed-integer nonlinear programming formulations for the two proposed approaches.
In Section~\ref{sec:reformulation} we present stronger reformulations of the models to improve their solvability.
%Then, we use
%mixed integer conic quadratic programming to handle the nonlinear fuel burn function.
%We first reformulate the original formulation with logical constraints using the well-known Big-M method. In Section~\ref{McCormick}, we introduce McCormick inequalities to reformulate the logical constraints.
We computationally test the proposed mathematical models using a real-world data of a major U.S. airline in Section ~\ref{Computation} and conclude with a few final remarks in Section ~\ref{Sec:6}.

\section{Problem Definition}
\label{Sec:2}
In this section, we briefly describe the problem setting. Consider a set of new flight pairs (i.e., consecutive flights, specifically a flight from the hub to a new demand point and its return flight to hub) to be accommodated into the existing flight schedule in near real time. An airline needs to accommodate new flights into an existing flight schedule without excessively disrupting the existing schedule. Therefore, we only shift the departure times of existing flights within the intervals already determined by the airline. Moreover, any arrival tardiness of the existing flights due to inserting new flights is penalized in the objective. In hub-and-spoke networks, connecting passengers represent
a non-negligible percentage of the total number of passengers. Therefore, we also respect the connection of passengers at the hub airport while optimizing the
departure and cruise times of flights. There are two cases for a new flight: (1) If the new flight is assigned to a larger aircraft, then this assignment may  capture more passengers, providing a greater revenue, but  compressing the cruise times of flights to accommodate new flights may increase the fuel cost significantly; (2) if the new flight is assigned to a smaller fuel efficient aircraft to reduce the fuel burn, then there may be an additional cost of spilled passengers with a decrease in revenue. To increase the profit for an airline, we introduce a second model which additionally includes aircraft swapping decisions.   \ignore{While a swap may reduce the fuel burn, if a flight is assigned to a smaller aircraft after the swap, then some of the passengers of the subsequent flights may be spilled due to the low capacity of the aircraft. In this study, we consider the interplay among all decision variables with the goal of maximizing the airline's profit.}

In the following sections, we first define the fuel burn as a function of cruise time and the penalty function for arrival tardiness of flights, and then provide a numerical example to show how to utilize the cruise time controllability and aircraft swapping to \ignore{accommodate} open-up enough space for the block times of new flights.

\subsection{Fuel and CO$_2$ emission cost function} One of the main contributions of this study is to increase the aircraft cruise speed to compensate for the time required to operate new flights. However, we need to consider the adverse effect of increasing cruise speed on fuel and carbon emissions costs.  To estimate the fuel burn, we use the cruise stage fuel flow model developed by the Base of Aircraft Data (BADA) project of EUROCONTROL \cite{europaram}. This model has been widely used in the literature.  The fuel burn (kg) of a flight as a function of its cruise time  $f$ (minutes) and its aircraft type $t$ can be calculated as follows:
%\todo{there is no need to refer to index i here}
\begin{equation}
F^t\left(f\right)= \alpha^{t}  \frac{1}{f}  + \beta^{t}  \frac{1}{{f}^2}  + \gamma^t  {f}^3 + \nu^t  {f}^2. \label{eq:cost} \nonumber
\end{equation}
The coefficients \noindent $\alpha^{t}, \beta^{t},  \gamma^t,  \nu^t  > 0$ %\todo{the coefficients are independent of the flight index i. aren't they? So we can drop the indices i.}
%\todo{Coefficients also include the information of flight distance between origin and destination airport. They depend on many parameters, but here I couldn't list all of them.}
%\todo{OK. Let's drop them in this section for simplicity of notation. We can keep subscript i's in the models.}
are expressed in terms of aircraft specific fuel consumption coefficients as
well as the mass of the aircraft, air density and gravitational
acceleration as provided in  \c{S}afak et al.\cite{Safak}. It is important to note that $F^t$ is a convex function whenever $f>0$. The minimizer of the fuel consumption function $F^t$ is represented by $u^t$, which is the ideal cruise time when an aircraft flies at the most fuel-efficient speed, referred to Maximum Range Cruise (MRC) speed. \ignore{In other words, MRC speed is the most fuel efficient speed of an aircraft.} Although the fuel burn is minimized at MRC speed, airlines may set higher cruise speed due to the scheduling constraints.

EUROCONTROL \cite{euro} states that each kg of fuel burn approximately produces 3.15 kg of CO$_2$ emission. Therefore, we can express fuel and CO$_2$ emission costs as a function of cruise time as follows:
\begin{equation}
c^t(f) = c_o F^t\left(f\right)%
%\left(c_o \right)  (\alpha_{i}^{t}  \frac{1}{f_i^t}+ \beta_{i}^{t}  %\frac{1}{{(f_i^t)}^2}+\gamma_i^t {(f_i^t)}^3+\nu_i^t {(f_i^t)}^2)
\end{equation}
\noindent where $c_o$ is the total cost of of fuel and CO$_2$ emitted by an aircraft per kg of fuel burned.

%\todo{I explain the deviation penalty function}

\subsection{Penalty function for arrival tardiness} In this paper, we aim to accommodate new flights with minimal
disruptions on the existing flight schedule. In this quest, we penalize the deviation from the original arrival times for the existing flights. Hoffman and Ball \cite{hoffman} suggest to use a nonlinear delay cost function to better reflect the reality since flight delay costs tend to grow with time at a greater rate than linear rate. Moreover, EUROCONTROL \cite{delay} performs a detailed investigation of airline cost functions and reports that the power curve provides a good fit to passenger costs as a function of delay duration. Therefore, we penalize the arrival tardiness with a convex increasing function of tardiness $b$ as

\begin{equation}
P(b) = \rho b^{\zeta} \label{eq:penaltyfunction}
\end{equation}

\noindent where $\rho \geq 0$. As in Akt{\"u}rk et al. (5), we let $\zeta = 1.5$ in our computational experiments.  \REV{{It is important to note that the US costs might differ from the ones provided in the EUROCONTROL \cite{delay}.}} 

Arrival tardiness can be reduced by compressing the cruise time of flights with an additional fuel burn cost. Therefore, the assignment of fuel efficient aircraft to new flights becomes more critical in order to reduce both delay and fuel burn costs.

\subsection{Numerical example} In this section, first we will provide a numerical example to show how the cruise time controllability can be utilized to accommodate new flights into an existing airline  flight schedule. Then, we will extend the example to show how aircraft swapping together with the cruise speed control can be used to achieve a more profitable schedule.

We give a sample schedule for two aircraft in Table \ref{schedule1}. Tail numbers of the aircraft and the flight numbers along with the origin and destination airports, planned departure times in local ORD time, planned block times and demand for the flights are listed in the table. Each aircraft visits ORD at least once in a day.  Let us now introduce new round-trip flights from ORD to MSP and back.

\begin{table}[htbp]\scriptsize
  \centering
    \caption{Original schedule.}
  \label{schedule1}%
    \begin{tabular}{l|cccccccc}
    \hline
    \hline
    & Tail \# & Flight \# & \multicolumn{1}{c}{From}  & \multicolumn{1}{c}{To}    & \multicolumn{1}{c}{Plan. Dep.} & \multicolumn{1}{c}{Plan. Dur.}  & \multicolumn{1}{c}{Plan. Arr.} &  \multicolumn{1}{c}{Demand} \\
    \hline
    &\multirow{4}*{N53442}		&1586	&ORD	&MCO	&08:00	&03:04 &11:04	&200 \\
Existing   & 			&633		&MCO	&ORD	&12:00	&03:07	&15:07	&180\\
flights   &			&451		&ORD	&IAH		&18:10	&03:08	&21:18	&190\\
  &			&584		&IAH		&ORD	&22:30	&02:50	&01:20	&186\\ \hline
  & \multirow{4}*{N45425}		&527		&ORD	&IAH		&08:45	&03:02	&11:47	&151\\
Existing   & 			&521		&IAH		&ORD	&12:32	&03:03	&15:35	&154\\
flights &			&623		&ORD	&MCO	&17:00	&03:02	&20:02	&160\\
&			&679		&MCO	&ORD	&21:10	&03:10	&00:20	&163\\
    \hline
  New &  		&1842	&ORD	&MSP	&13:15	&01:40	&14:55	&183\\
    flights &			&430		&MSP	&ORD	&16:10	&01:45	&17:55	&168\\
    \hline
    \hline
    \end{tabular}%
\end{table}%

Figure \ref{fig:origschedule} gives the time-space network representation of the original schedule. The red and blue arcs in Figure \ref{fig:origschedule} represent routes for aircraft N53442 and N45425 respectively. The flight arcs originate from the departure airport at the planned departure time and end at the destination airport after the planned block time. Ground arcs represent the aircraft turnaround times needed to prepare the aircraft for the next flight.

\begin{figure}[htbp]
\centering
\includegraphics[width=0.7\textwidth]{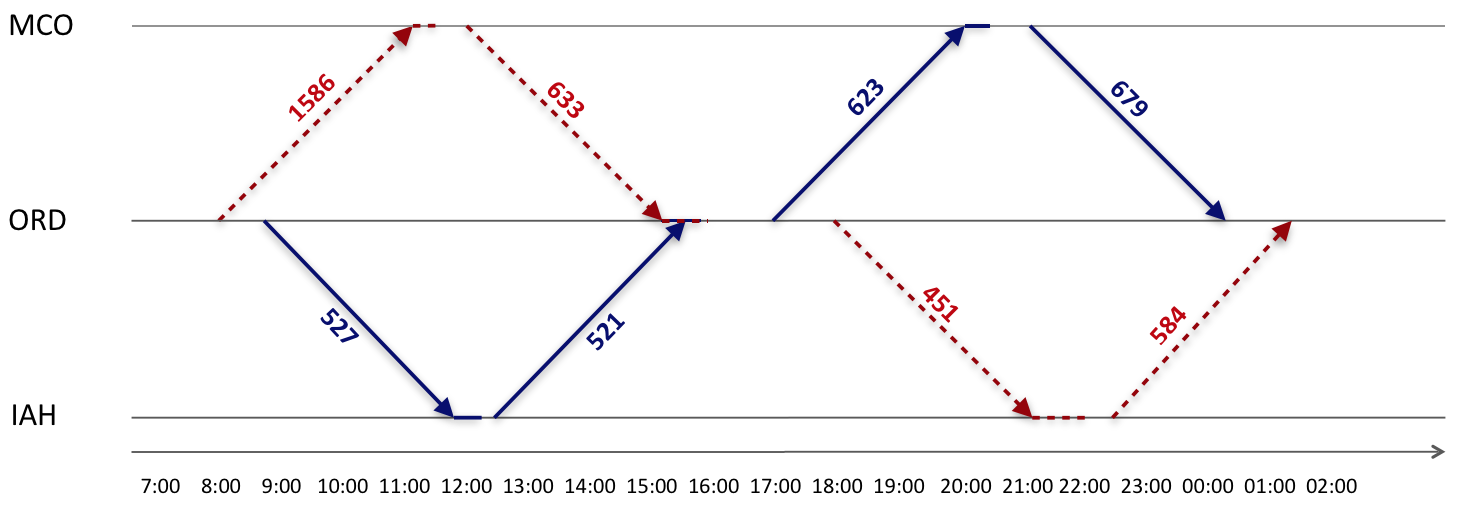}
\caption{Original time-space network.} \label{fig:origschedule}
\end{figure}

In this example, we assume that aircraft N53442 is a Boeing 767-300 and N45425 is an Airbus 320-212. The number of seats of the aircraft are 218 and 180, respectively. In the original schedule we assume that the aircraft fly at the most fuel efficient speed (MRC speed) and estimate the fuel burn rates as 87 kg/min and 40 kg/min for aircraft N53442 and N45425, respectively. The fuel burn rate is calculated using the fuel flow model of BADA as mentioned above. We assume that for each flight, non-cruise stages take 30 minutes. Then, cruise stages take 30 minutes less than flight block times given in Table \ref{schedule1}. Assuming $c_{fuel}$ = 1.2 \$/kg and $c_{CO_2}$ = 0.2 \$/kg, the total cruise stage fuel and carbon cost for the original schedule is \$100,593.

Let us assume that an airline wants to operate two new flights 1842 and 430. In order to open up sufficient time to accommodate these new flights, one approach is to compress the cruise times and adjust the departure times simultaneously. We refer to this approach with cruise time controllability as CTC. If the airline wants to meet the passenger demand of 183 for the new flight 1842, the only way to do so is to assign the new flights to aircraft N53442. These new flights can be placed between flights 633 and 451. The necessary block times for the new flights can be made available by left-shifting flights 1586 and 633 and right-shifting flights 451 and 584. However, we assume that the airline wishes to keep such a new schedule close to the original one. Thus, we only allow departure times to deviate at most 90 minutes from the planned departure times in the original schedule. In addition, if a flight arrives 15 minutes later than the original planned arrival times, we penalize the tardiness with a nonlinear function in \eqref{eq:penaltyfunction}. We let $\rho=5$ and $\zeta=1.5$ in the penalty function. Because of these scheduling limitations,
\ignore{operate business trips in the morning. Thus, we only allow departure times of morning flights to deviate 30 minutes from the planned departure times in the original schedule. In addition, we assume that airline wants to depart all flights before 11:00 pm in local times of departure airports. Therefore, }the new fights cannot be accommodated by only shifting the flight departure times. We also need to compress the cruise times of the existing flights 633, 451 and 584 by 17, 17 and 3 minutes, respectively. Besides, the cruise times of new flights 1842 and 430 are compressed by 8 minutes to satisfy the scheduling restrictions. However, flight 451 arrives 37 minutes later then the original planned arrival time, thus 22 minutes of this tardiness are penalized by \$552. In Figure \ref{fig:schedule1}, we give the time-space network representation of resulting schedule. In this figure, the dotted arcs represent the flights with original block times, whereas the line arcs represent the flights with compressed cruise times.

\begin{figure}[htbp]
\centering
\includegraphics[width=0.7\textwidth]{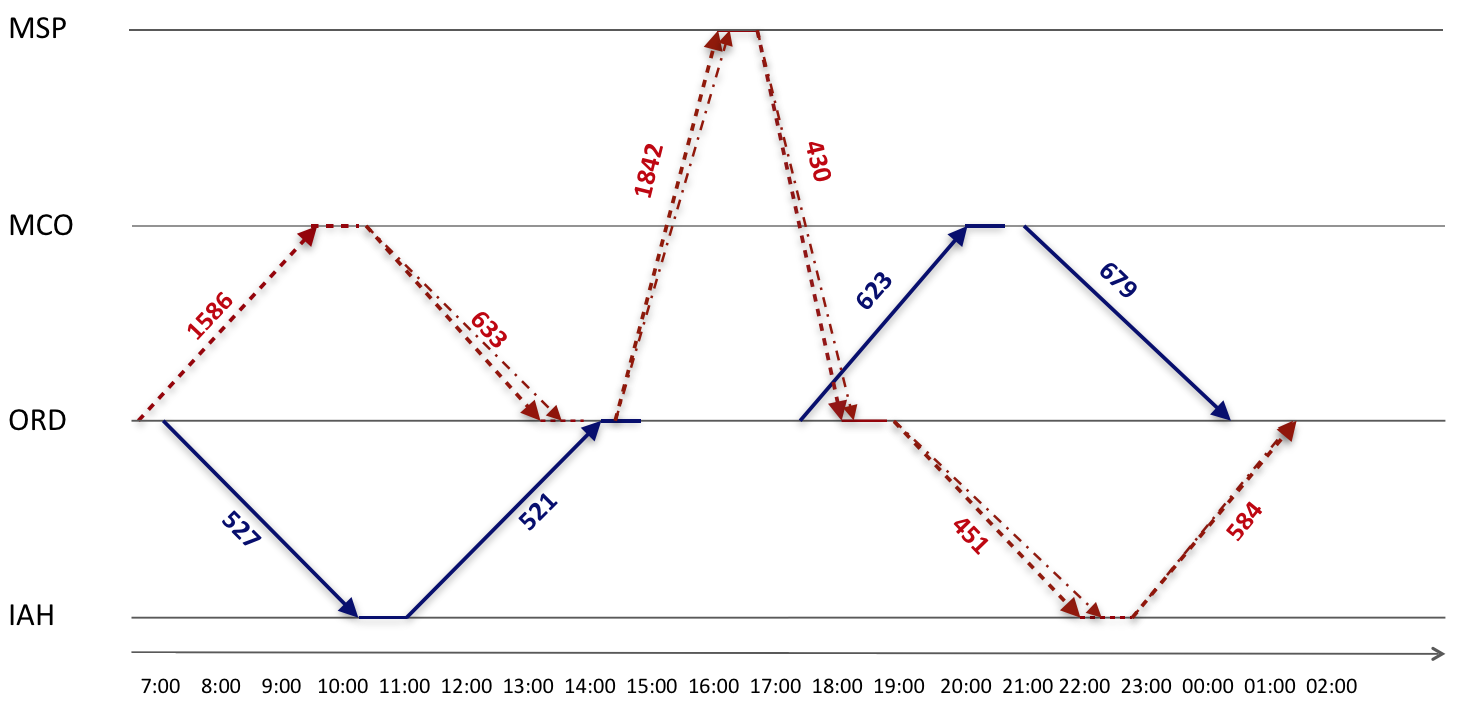}
\caption{Time-space network with CTC.} \label{fig:schedule1}
\end{figure}

Compressing the cruise times of the flights incurs additional costs of fuel burn and CO$_2$ emission. The total fuel burn and CO$_2$ emission costs of the
existing flights increases to \$101,286. Moreover, for new  flights, the total
fuel burn and CO$_2$ emission cost is \$16,044. Therefore, the fuel and emission cost increment compared to the original schedule is \$16,737 calculated as \$16,737 = \$101,286 + \$16,044 - \$100,593. To reduce the fuel burn by reassigning the aircraft, we also propose an aircraft swapping mechanism together with the cruise time controllability, referred to as CTC-AS.

\begin{figure}[htbp]
\centering
\includegraphics[width=0.7\textwidth]{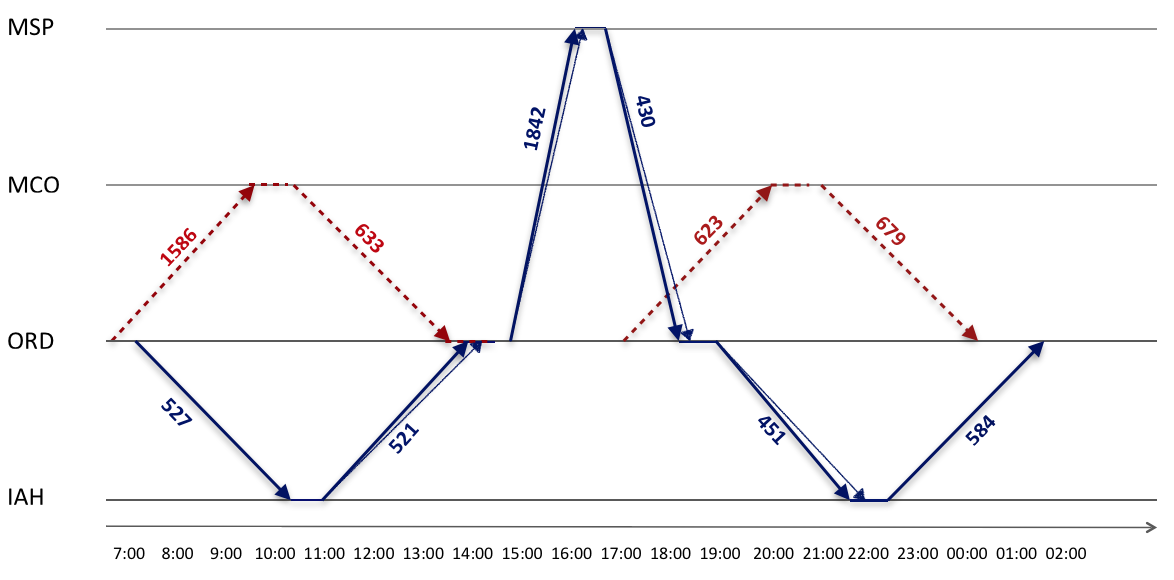}
\caption{Time-space network with CTC-AS.} \label{fig:schedule2}
\end{figure}

In the CTC-AS approach, we swap the aircraft of flights 451 and 623 before departure. In the new schedule, aircraft N53442 operates flights 1586, 633, 623, and 679, whereas aircraft N45425 operates 527, 521, 451, and 584. We provide the time-space network representation of the resulting schedule in Figure \ref{fig:schedule2}. To reduce the fuel expenses, the new flights are assigned to the fuel efficient aircraft N45425. However 10, 6 and 3 passengers of flights 451, 584 and 1842, respectively are spilled due to the low capacity of the aircraft N45425, thus resulting in a cost of \$2,007. \ignore{However, 10 passengers of flight 451, 6 p  and 3 passengers of the new flight 1842 are spilled due to insufficient seat capacity of the aircraft N45425, thus resulting in \$3,514. }
%\todo{Is 3514 the cost of spilling 30 passengers? Seems too low.} \todo{Spill cost parameter is an experimental design factor in the computation. Here, in this small instance, I used the lower cost parameter. Later, in the computations, I compare low and high settings. The higher level is almost three times of the low level.} \todo{Referees may raise this issue.}
Similarly, passengers' revenue obtained from ticket sales of new flights are reduced to \$49,923 from \$49,548, since 3 passengers of new flights are not served due to low seat capacity. An additional \$500 cost of swapping is incurred in the new schedule. On the other hand, the savings from the fuel burn and CO$_2$ emission costs may compensate for these additional cost of spilled passengers, revenue loss and cost of swapping. Indeed, the
fuel expenses and CO$_2$ cost of the new
flights are reduced to \$8,118, almost
half of the fuel expenses of the CTC approach. Even though the total fuel burn and CO$_2$
emission cost of existing flights slightly increases to \$101,590, the fuel and emission cost increment significantly decreases from \$16,737 to \$9,115. Penalization cost for arrival tardiness of flight 451 also decreases to \$115. Therefore, resulting schedule improves the airline's profit from \$30,234 to \$35,412.

We give the operational cost components and revenues of two schedules achieved by CTC and CTC-AS approaches in Table \ref{eq:CostPublished}. Airline's profit is calculated as follows:
\begin{eqnarray}
\text{Profit } & = & \text{Revenue} - \left(\text{Fuel \& Emiss. Cost Increment}\right) - \text{Spilled Cost}  \nonumber \\
& & -  \text{Penalty Cost} - \text{Swap Cost} -  \left(\text{Crew \& Service Costs}\right). \nonumber
\end{eqnarray}

\begin{table}[ht]
		\setlength{\tabcolsep}{1pt}
  \centering
  \caption{Cost comparison.}
  \label{eq:CostPublished}%
  \scalebox{0.9}{
\begin{tabular}{c r r r  r r   }
\hline
\hline
\multicolumn{1}{r}{} &\multicolumn{1}{r}{Fuel \& Emiss.}  &\multicolumn{1}{r}{Spilled}& \multicolumn{1}{r}{Deviation}& \multicolumn{1}{r}{Swap} \\
\multicolumn{1}{r}{} &\multicolumn{1}{r}{Cost Increment ($\$$)}&\multicolumn{1}{r}{ Cost ($\$$)} &\multicolumn{1}{r}{ Penalty ($\$$)} & \multicolumn{1}{r}{Cost ($\$$)} \\
\hline
 \multicolumn{1}{l}{CTC} &16,044           &0	&552.0  &0			\\

\multicolumn{1}{l}{CTC-AS} &8,118   &2,007	&115.0    &500		\\
 \hline  \hline
& \multicolumn{1}{r}{Crew \& Service}  &  \multicolumn{1}{r}{Passengers}&\multicolumn{1}{r}{} \\
& \multicolumn{1}{r}{Cost ($\$$)}   & \multicolumn{1}{r}{Revenue ($\$$)}&\multicolumn{1}{r}{Profit ($\$$)}\\
  \hline
\multicolumn{1}{l}{CTC}&4,400  &49,923  &30,234\\

\multicolumn{1}{l}{CTC-AS} &4,400	 &49,548   &35,412\\
  \hline  \hline
    \end{tabular}
}
\end{table}

\section{Mathematical Formulations}
\label{Sec:3}
In this section, we present the mathematical formulations of the two approaches described in the previous section. We start with the simpler CTC model that adjusts the departure times and controls the cruise time, and then extend it to CTC-AS by incorporating aircraft swapping as well.

\subsection{Formulation with cruise time controllability}
\label{FirstFormulation}
We first give a list of sets, parameters and decision variables used in the model.\\

\noindent
\textbf{Sets:}
\begin{longtable}{p{.07\textwidth}p{.9\textwidth}p{.18\textwidth}}
$E$&set of existing flights in the schedule\\
$E_{O}$&set of existing outbound flights from the hub \\
$E_{I}$& set of existing inbound flights arriving to the hub \\
$N$&set of new flights \\
$N_{O}$&set of new outbound flights from the hub to a new demand point\\
$N_{I}$&set of new inbound flights from a new demand point to the hub\\
$T$&set of aircraft types\\
$C_E$&set of pairs of existing consecutive flights of the same aircraft, $(i,j), \ i \in E, j \in E$ \\
$C_N$&set of pairs of new consecutive flights of the same aircraft, $(i,j), \ i \in N_{O}, j \in N_{I}$\\
$U_i$& set of flights that can follow flight $i$, $i \in E \cup N_{I}$\\
$G_i$&set of existing outbound flights which have passenger connections from flight $i \in E_{I}$ \\
\end{longtable}

%\todo{$C -> C_E; P -> C_N$}

\noindent
\textbf{Parameters:}
\begin{longtable}{p{.07\textwidth}p{.9\textwidth}p{.18\textwidth}}
$\chi_i^t$& 1 if aircraft type $t \in{T}$ is originally assigned to existing flight $i \in{E}$, and 0 o.w.\\
$\left[\ell_{i}^{t},u_{i}^{t}\right] $&time window for the cruise time of flight $i \in{E \cup N}$ with aircraft type $t \in{T}$\\
$\left[d_i^{\ell},d_i^{u}\right]$&time window for the departure of  flight  $i \in{E \cup N}$ \\
$\eta_i$&non-cruise time of flight  $i \in{E \cup N}$  \\
$\tau_{i}^t$&turnaround time needed to prepare aircraft type $t \in{T}$ for the connection after flight $i \in E  \cup N$ \\
$\lambda_{ij}$&time needed for the connection of passengers from flight $i \in E_{I}$ to flight $j \in G_{i}$ \\
$\kappa^t$&number of seats of an aircraft type $t \in{T}$\\
$\mu_i$&number of passenger demand of each new flight $i \in{N}$  \\
$t(i)$&aircraft type assignment of existing flight $i \in E$ \\
$\pi_i$ &ticket price of new flight $i \in{N}$ \\
$\sigma_{i}$&cost of spilled passengers of new flight $i \in{N}$\\
$\phi$&crew and serving cost for new flights\\
%$\REV{\rho}$&\REV{coefficient of penalty function for deviations in arrival times of existing flights}\\
$a^{o}_i$&original arrival time of flight $i \in E$
\end{longtable}

\noindent
\textbf{Decision variables:}
\begin{longtable}{p{.07\textwidth}p{.9\textwidth}p{.48\textwidth}}
$f_i^t$&cruise time of flight $i \in E \cup N$ with aircraft type $t \in T$\\
$d_i$& departure times of flight $i \in E \cup N $ \\
$a_i$&arrival time of flight $i$ to its destination $i \in E \cup N$\\
$b_i$&deviation from the original arrival time of flight $i \in E$\\
$z_{i}^t$& 1 if aircraft type $t \in{T}$ is assigned to flight $i \in{N}$, and 0 o.w.\\
$y_{ij}$& 1 if flight $i \in E \cup N_{I}$ is followed by flight $j \in U_i$, and 0 o.w.
\end{longtable}

%\todo{In fact, cruise time decision for existing flights does not depend on the aircraft type $t$ in CTC model ($f_i^{t(i)}$). Then, in CTC-AS model, it depends on the aircraft types, since we change the original aircraft assignment by swapping. How should I emphasize this difference in the definition or notation of cruise time variable. It directly makes a difference in fuel cost function as well.} \todo{This is fine. Not to worry}.

%\todo{Why do we need to define two types of variables (u and y)? Isn't it sufficient to define $y_{ij}$ = 1 if flight j follows flight i, 0 o.w. and define the sets for these indices appropriately for constraints (3)--(7) for example? I will consider}

In the new schedule, an existing flight $i \in E$ can be followed by a new outbound flight $j \in N_{O}$.  Similarly, each new inbound flight $i \in N_{I}$ can be followed by an existing flight $j \in E$.

For each $i \in E \cup N, t \in T$, we redefine the fuel and CO$_2$ emission cost function as
\begin{equation*}
c_i^t(f_i^t) =
\begin{cases}
c_o  \bigg (\alpha_{i}^{t}  \frac{1}{f_i^t}+ \beta_{i}^{t}  \frac{1}{{(f_i^t)}^2}+\gamma_i^t {(f_i^t)}^3+\nu_i^t {(f_i^t)}^2 \bigg) & \text{if } z_i^t=1 \\
0 & \text{if }z_i^t=0,
\end{cases}
\end{equation*}
so that if aircraft type $t$ is not assigned to flight $i$, then $c_i^t(f_i^t) = 0$.

%Given set of flights along with the aircraft routes, the model schedules all existing flights to open up a time space in the aircraft routes to accommodate new flights. The goal is to determine new departure times, cruise times and aircraft type assignment to new flights to maximize airline's profit. The profit is the revenue minus operational costs such as incremental fuel and CO$_2$ emission costs and spilled passenger costs. In this model, a new flight can be accommodated in an aircraft route in three ways. If a new flight is the first flight of an aircraft route, then there exists an existing flight $i$ which follows the new flight, i.e., $y_{ni}=1$ for $n \in N_{I}, i \in E$. If a new flight is the last flight of an aircraft route, then an existing flight is followed by the new flight, i.e., $y_{in}=1$ for $i \in E, n \in N_{O}$. Otherwise, new flight pair $(n,m)$ is added between two consecutive flights $(i,j)$ satisfying that $y_{in}=u_{mj}=1$ for $(i,j) \in C_E, (n,m) \in C_N$.

Using the notation above, we now provide a mathematical model of the problem (CTC):  %\todo{Please remove extra quads to fit the formulation into the pagewidth. The journal has two columns; so we need to keep the lines short.}
\label{mathmodel}

\allowdisplaybreaks

\begin{align}
\text{max} \ &  \sum_{t \in T} \sum_{i \in N} \pi_i  \left( min \left(\mu_i, \kappa^t \right) \right) z_i^t  -\sum_{i \in E} \left(c_i^{t(i)}(f_i^{t(i)})
- c_i^{t(i)}(u_i^{t(i)}) \right)  \nonumber \\
& - \sum_{t \in T} \sum_{i \in N} c_i^{t}(f_{i}^{t})
- \sum_{t \in T} \sum_{i \in N}   \sigma_{i}  max (0, \mu_i - \kappa^t)  z_i^t  -\sum_{i \in E} \rho b_i^{\zeta}- \phi   \nonumber
\end{align}
\begin{align}
\text{s.t.} &   \sum_{i \in E_{I}} y_{in} \leq 1,  && n \in N_O  \label{eq:4}\\
& \sum_{i \in E_{O}} y_{ni} \leq 1,  && n \in N_I \label{eq:5} \\
& \sum_{i \in E_{I}} y_{in} + \sum_{i \in E_{O}} y_{mi}     \geq 1, && (n, m) \in C_N  \label{eq:6} \\
& \sum_{n \in N_{O}} y_{in}  \leq 1,   &&  i \in E_{I}  \label{eq:4a} \\
& \sum_{m \in N_{I}} y_{mi} \leq 1,  &&  i \in E_{O} \label{eq:4b} \\
& y_{in}  = y_{mj},  &&  (i, j) \in C_E, (n, m) \in C_N \label{eq:28}\\
&  | z_n^t - \chi_i^t | \leq (1 - y_{in} ),    &&  i \in E_{I}, n \in{N_O}, t \in T \label{eq:7} \\
&    | z_{n}^t - \chi_i^t | \leq (1 - y_{ni} ),   &&  i \in E_{O}, n \in N_I, t \in T \label{eq:9} \\
&    z_{n}^t = z_m^t,  &&  (n, m) \in C_N, t \in T \label{eq:8} \\
& \sum_{t \in T} z_{n}^t = 1,  && n \in N_O \label{eq:27} \\
& d_i + f_i^{t(i)} + \eta_i = a_i, && i \in E \label{eq:1}\\
& d_n + \sum_{t \in T} f_i^t + \eta_n = a_n, && n \in N \label{eq:2}\\
& a_i + \lambda_{ij} \leq d_j,  && i \in E_{I}, j \in G_{i} \label{eq:3}\\
&   a_n + \sum_{t \in T} \tau_{n}^t  z_n^t \leq d_{m}, && (n, m) \in C_N \label{eq:10}\\
&  \text{If } y_{in} =1   \text{ then }    a_i +\tau_{i}^{t(i)} \leq d_n, &&  i \in E_{I}, n \in N_O \label{eq:16}  \\
& \text{If } y_{ni} =1,  \text{ then }   a_n + \tau_{n}^{t(i)}  \leq d_i,  && i \in E_{O}, n \in N_I  \label{eq:15} \\
& \text{If }\sum_{n \in N_O} y_{in} =0,  \text{ then }  a_i +  \tau_{i}^{t(i)}   \leq d_{j}, && (i, j) \in C_E \label{eq:17}\\
&a_i - \left(a^{o}_i +15\right) \leq b_i  && i \in E  \label{eq:deviation}\\
& \ell_i^{t(i)} \leq f_i^{t(i)}  \leq   u_i^{t(i)}, &&  i \in E \label{eq:18} \\
&  \ell_i^{t} z_i^t \leq f_{i}^{t}  \leq  u_i^{t} z_i^t, &&  i \in N, t \in T  \label{eq:19} \\
&  d_i^{\ell} \leq d_i \leq d_i^{u},  && i \in E \cup N   \label{eq:26} \\
%& y_{in} = 0, \quad\quad\quad\quad\quad\qquad\qquad\quad\quad\quad\;\;\,\, i \notin E_{I}, n \in N_O \label{eq:24.a} \\
%& y_{ni} = 0, \quad\quad\quad\quad\quad\qquad\qquad\quad\quad\quad\;\;\; i \notin E_{O}, n \in N_I \label{eq:25.a} \\
& z_i^t \in \{0,1\},  && i \in N, t \in T \label{eq:23}\\
& y_{ij}\in \{0,1\},  && i \in E \cup N_{I}, j \in U_i \label{eq:24}\\
& b_i \geq 0, && i \in E \label{eq:deviation1}
\end{align}

For given aircraft routes, we aim to generate a new flight schedule to introduce new flights with the goal of maximizing airline's profit. The first term of the objective is the revenue from ticket sales. The number of tickets sold can be determined as the minimum of passenger demand and seat capacity of the assigned aircraft. The remaining terms in the objective function represent the operational costs. The second term is the  incremental cost of fuel burn associated with speeding up the aircraft of existing flights. Similarly, the third term represents the total cost of fuel burn and carbon emissions for new flights. The fourth term is the cost of spilled passengers due to insufficient seat capacity of assigned aircraft to new flights. The fifth term is a penalty for arrival tardiness of the existing flights. Finally, the sixth term is the total costs of crew and service for new flights including landing fees, ground handling fees, insurance fees, onboard services costs, etc.

Constraint \eqref{eq:4} ensures that new outbound flight $n$ follows at most one existing flight $i$ arriving to the hub airport. Similarly, constraint \eqref{eq:5} guarantees that new inbound flight $n$ is followed at most one existing flight $i$ departing from the hub airport. Constraint \eqref{eq:6} assures that new flight pair $(n, m)$ is covered by an aircraft route. Constraints \eqref{eq:4a}--\eqref{eq:4b} ensure that an existing flight does not follow or is immediately followed by two different new flights. Constraint \eqref{eq:28} keeps the sequence of existing flights as in the original schedule. If a new flight pair $(n,m)$ is operated between an existing flight pair $(i,j)$, then the model ensures that $y_{in}=y_{mj}=1$. Otherwise, $y_{in}=y_{mj}=0$ for $(i,j) \in C_E$.

Constraints \eqref{eq:7}--\eqref{eq:27}, determine the aircraft type assignment to a new flight pair $(n, m)$. If $y_{in} = 1$ or $y_{ni} = 1$, then the corresponding aircraft of existing flight $i \in E$ is assigned to new flight pair $(n, m)$. In the schedule, we only allow new flight $n$ to depart from the hub and be immediately followed by a return flight $m$ so that we make the same fleet assignment to flights $n$ and $m$ in constraint \eqref{eq:8}. A new flight is assigned to one fleet type in constraint \eqref{eq:27}.

Constraints \eqref{eq:1}--\eqref{eq:2} define the arrival time of the flights to their destination airport. Constraints \eqref{eq:3} ensure the minimum time requirement for the passenger connections. Similarly, constraints \eqref{eq:10}--\eqref{eq:17} maintain the precedence relations among the flights assigned to the same aircraft in the new schedule. For a new flight pair $(n,m)$, constraint \eqref{eq:10}
guarantees the minimum time requirement for aircraft connections between flights $n$ and $m$. If an existing flight $i$ follows the new flight $n$, then constraint \eqref{eq:16} ensures that flight $n$ does not depart before the arrival time of flight $i$ plus its turnaround time.    If a new flight $n$ follows an existing flight $i$, then constraint \eqref{eq:15} enables incoming aircraft of flight $n$ to catch flight $i$. On the other hand, if no new flight is scheduled between existing flights $i$ and $j$, then constraint \eqref{eq:17} keeps the minimum aircraft turnaround time between flights $i$ and $j$ as in the original schedule. Constraints \eqref{eq:deviation} determine the arrival tardiness for new flights. In this study, we do not penalize the first 15 minutes of tardiness as it is a common notion in practice.

Constraints \eqref{eq:18}--\eqref{eq:19} apply cruise time upper and lower bounds for each flight, respectively. Constraint \eqref{eq:26} defines the time intervals for the departures of both existing and new flights. For instance, due to time sensitivity of the business trips, one can allow
departures in the morning within certain time intervals, which have already
been determined by the airline. The rest of the constraints \eqref{eq:23}--\eqref{eq:deviation1} define the domain of decision variables.

An important feature of the proposed mathematical formulation is that
the problem is formulated without keeping track of individual aircraft. The decision variable $y$ denotes on which flights are operated before/after the new flights. Since we also keep the sequence of existing flights operated by the same aircraft as it was in the sample schedule, following the index of $y$ decision determines the route of each aircraft. In the model, the aircraft tail of existing flights are given and not changed. This is a great advantage so that the proposed model determines the aircraft tail assignment with less computational effort. Note that the aircraft type information is also necessary for the cost calculations in the objective function.

The proposed formulation is a mixed-integer  optimization model with nonlinear cost terms in the objective function and logical constraints \eqref{eq:16}, \eqref{eq:15}, and \eqref{eq:17}. \REV{{In Section \ref{sec:reformulation}, we present the logical constraints mathematically using the Big-M method and McCormick inequalities, respectively.  These reformulations enable us to solve relatively large instances to optimality very efficiently.}}

\subsection{Formulation with cruise time controllability and aircraft swapping}
\label{SecondFormulation}
In this section, we additionally include an option of swapping aircraft in the model. Although the ability of swapping the aircraft provides a greater flexibility to make time spaces for new flights with an increased profit, several challenges arise. First, additional binary aircraft assignment decisions for the existing flights are required. Second, there exists a trade-off between the fuel burn and number of spilled passengers of existing flights, since swapping the aircraft of a flight with a more fuel efficient but smaller aircraft not only decreases the fuel burn but may also spill some of the passengers. Third, cruise time decisions of the existing flights depend on the aircraft assignments due to fuel burn as a function of aircraft type.

In order to include an option of aircraft swapping, we first define additional sets and parameters, and redefine the decision variables.\\

\noindent
\textbf{Sets \& Parameters:}
\begin{longtable}{p{.07\textwidth}p{.9\textwidth}p{.48\textwidth}}
$R$& set of aircraft routes in the original schedule \\
$E_r$& set of flights in each aircraft route $r \in R$ \\
$S(i)$& set of possible flights whose aircraft can be swapped with the aircraft of flight $i \in E$ \\
$p(i)$&predecessor of flight $i \in E$\\
$\sigma_{i}$&cost of spilled passengers of flight $i \in{E \cup N}$ \\
$\psi$ &cost of swapping an aircraft
\end{longtable}

\noindent
\textbf{Decision variables:}
\begin{longtable}{p{.07\textwidth}p{.9\textwidth}p{.48\textwidth}}
$z_i^t$& 1 if aircraft type $t \in{T}$ is assigned to flight $i \in{E \cup N}$, and 0 o.w.\\
$s_{ij}$& 1 if the aircraft of flight $i \in E$ and flight $j \in S(i)$ are swapped at their destination and 0 o.w.
\end{longtable}

Then, the mathematical formulation that includes the option of the aircraft swapping  (CTC-AS) is stated as follows: %\todo{Please remove extra quads to fit the formulation into the pagewidth. The journal has two columns; so we need to keep the lines short.}
\allowdisplaybreaks
\ignore{
\begin{align}\text{max} \quad & \sum_{t \in T} \sum_{i \in N} \left( min \left(\mu_i, \kappa^t \right) \right)  z_{i}^t  p_i  \tag{CTC-AS}\nonumber \\
 & -\sum_{i \in E} \left( \sum_{t \in T} c_i^{t}(f_i^t)  - c_i^{t(i)}(u_i^{t(i)}) \right) \nonumber - \sum_{t \in T} \sum_{i \in N} c_i^{t}(f_i^t) \nonumber \\
& - \sum_{t \in T} \sum_{i \in E \cup N} \left( max \left(0, \mu_i - \kappa^t \right) \right)  z_{i}^t  \sigma_{i} - \sum_{i \in E} \sum_{j \in S(i)} \left(\psi s_{ij}\right)/2 \label{eq:00.1} \\
 \text{subject to} \nonumber
\end{align}
}
\begin{align}
\text{max}  & \sum_{t \in T} \sum_{i \in N} \pi_i   \text{ min} \left(\mu_i, \kappa^t \right)  z_{i}^t  \nonumber  %\\% \tag{CTC-AS}\nonumber \\
 -\sum_{i \in E} \left( \sum_{t \in T} c_i^{t}(f_i^t)  - c_i^{t(i)}(u_i^{t(i)}) \right) \\ &\nonumber - \sum_{t \in T} \sum_{i \in N} c_i^{t}(f_i^t) \nonumber
 - \sum_{t \in T} \sum_{i \in E \cup N} \sigma_{i} \text{ max} \left(0, \mu_i - \kappa^t \right)   z_{i}^t  \nonumber \\
 &  -\sum_{i \in E} \rho b_i^{\zeta}  - \sum_{i \in E} \sum_{j \in S(i)} \left(\psi/2\right) s_{ij} - \phi  \nonumber
 \end{align}
\begin{align}
& \text{s.t.} \quad \text{Constraints} \quad \eqref{eq:4}-\eqref{eq:4b} && \nonumber \\
& | z_n^t - z_i^t | \leq (1 - y_{in} ), &&  i \in E_{I}, n \in{N_{O}}, t \in T \label{eq:5.1} \\
& | z_{n}^t - z_i^t | \leq (1 - y_{ni} ),  &&  i \in E_{O}, n \in N_{I}, t \in T \label{eq:6.1} \\
& z_{n}^t = z_{m}^t,  &&  (n, m) \in C_N, t \in T \label{eq:7.1} \\
& \sum_{t \in T} z_{n}^t = 1,   &&  n \in N_{O} \label{eq:8.1} \\
& | z_{p(i)}^t - z_i^t | \leq \sum_{j \in S(i)} s_{ij},  && t \in T, i \in E \label{eq:9.1} \\
& z_j^{t(p(i))} \geq s_{ij,} &&  i \in E, j \in S(i) \label{eq:11.1} \\
&  \sum_{i \in E_r} \sum_{j \in S(i)} s_{ij} \leq 1,  && r \in R \label{eq:12.1} \\
& s_{ij} = s_{ji}, && i \in E, j \in S(i) \label{eq:13.1} \\
& |  y_{in} - y_{mj} | \leq \sum_{k \in S(j)} s_{jk}, &&  (i, j) \in C_E, (n, m) \in C_N \label{eq:14.1}\\
& |y_{in}- y_{mk} | \leq 1 - s_{jk},  &&   k \in S(j), (i, j) \in C_E, \nonumber \\
&	&& (n, m) \in C_N   \label{eq:15.1}\\
& d_i + \sum_{t \in T} f_i^t + \eta_i = a_i, && i \in E \cup N \label{eq:4.1}\\
&   a_n + \sum_{t \in T} \tau_{n}^t  z_n^t \leq d_{m},  && (n, m) \in C_N \label{eq:17.1}\\
&  \text{If } y_{in} =1,   \text{then }      a_i + \sum_{t \in T} \tau_{i}^t  z_i^t  \leq d_n, &&  i \in E_{I}, n \in N_{O} \label{eq:16.1}  \\
& \text{If } y_{ni} =1,  \text{then }   a_n \! + \!\sum_{t \in T} \tau_{n}^t  z_{n}^t   \leq d_i,  &&  i \in E_{O}, n \in N_{I}  \label{eq:18.1} \\
&  \text{If } \sum_{k \in S(j)} \! s_{jk} = 0, \text{then }  a_i \! +  \! \sum_{t \in T} \tau_{i}^t  z_i^t   \leq d_{j}, && (i, j) \in C_E \label{eq:19.1}\\
%&  \text{If }\sum_{n \in N_{O}} \! y_{in} + \! \! \sum_{k \in S(j)} \! s_{jk} = 0, \text{then }  a_i \! +  \! \sum_{t \in T} \tau_{i}^t  z_i^t   \leq d_{j}, \;  (i, j) \in C_E \label{eq:19.1}\\
&   \text{If } s_{jk}  \! = 1, \text{then }   a_i \! + \!  \sum_{t \in T} \tau_{i}^t  z_i^t   \leq  d_{k}, && k \in S(j), (i, j) \in C_E \label{eq:20.1} \\
%&   \text{If } s_{jk}  \! - \! \! \! \sum_{n \in N_{O}} \! \! y_{in} = 1, \text{then }   a_i \! + \!  \sum_{t \in T} \tau_{i}^t  z_i^t   \leq  d_{k},  k \in S(j), (i, j) \in C_E \label{eq:20.1} \\
%&& \qquad\qquad\qquad\qquad\qquad\qquad\;\; k \in S(j), (i, j) \in C_E %\label{eq:20.1}\\
& a_i + \lambda_{ij} \leq d_j,  && i \in E_{I}, j \in G_{i} \label{eq:3.2}\\
&a_i - \left(a^{o}_i +15\right) \leq b_i && i \in E  \label{eq:newdeviation}\\
& \ell_i^{t} z_i^t \leq f_i^t,  \leq  u_i^{t} z_i^t  &&  i \in E \cup N, t \in T  \label{eq:21.1} \\
& d_i^{\ell} \leq d_i \leq d_i^{u}, && i \in E \cup N  \label{eq:22.1} \\
%& y_{in} = 0,  \qquad\qquad\qquad\qquad\quad\;\,  i \notin E_{I}, n \in N_{O} \label{eq:23.1} \\
%& y_{ni} = 0,  \qquad\qquad\qquad\qquad\quad\;\,  i \notin E_{O}, n \in N_{I} \label{eq:23.1.1} \\
& z_i^t \in \{0,1\}, && i \in E \cup N, t \in T \label{eq:24.1}\\
& y_{in}\in \{0,1\},  &&   i \in E \cup N_{I}, n \in U_i \label{eq:25.1}\\
& s_{ij} \in \{0,1\},  && i \in E, j \in S(i) \label{eq:27.1}\\
& b_i \geq 0,  && i \in E \label{eq:newdeviation1}
\end{align}

The objective function of the second model is slightly different than the objective of the first model. If the aircraft of a flight is swapped with a smaller aircraft, then some of the passengers of the subsequent flights might be spilled. Therefore, we include an additional cost term for spilled passengers of the existing flights. We also add a new swap cost term to cover the cost of changes caused by swapping the aircraft. The rest of the objective terms are same as the first model.
Despite the additional cost of spilled passengers and swapped aircraft, the CTC-AS introduces potential for greater profit by reducing the fuel burn.

We use the same constraints \eqref{eq:4}--\eqref{eq:4b} of the first formulation to accommodate new flight pairs in an aircraft route. However, the aircraft type assignment constraints \eqref{eq:5.1}--\eqref{eq:8.1} are slightly different.
 If new flight $n$ follows an existing flight $i$, then aircraft type assignments of flights $i$ and its immediate successor $n$ will be same per constraint \eqref{eq:5.1}. Similarly, if new flight $n$ is followed by an existing flight $i$, then constraint \eqref{eq:6.1} assigns the same aircraft type to flights $n$ and $i$. Constraint \eqref{eq:7.1} ensures that the aircraft type assignments of the consecutive new flights are same. Constraint \eqref{eq:8.1} assigns exactly one aircraft type to each new flight pair.

Constraints \eqref{eq:9.1}--\eqref{eq:11.1} relate aircraft swap decisions to assignment decisions. If aircraft of flight $i$ is not swapped with another one before its departure, then the aircraft type assignment of flight $i$ and its predecessor flight $p(i)$ will be the same. In other words, if there is no swap before the departure of flight $i$, then $s_{ij}=0$ for all flights $j \in S(i)$. Therefore, $z_i^t=z_{p(i)}^t$ for each aircraft type $t$ per constraint \eqref{eq:9.1}. Otherwise,  i.e., $s_{ij}=1$, then the aircraft type assignment of flight $j$ and the predecessor flight $p(i)$ of flight $i$ will be the same. That is, flight $j$ is taken over by the aircraft of the predecessor of $i$,  in the original schedule. Aircraft type assignment of flight $j$ is modeled in constraint \eqref{eq:11.1}.  Constraint \eqref{eq:12.1} limits the number of swaps on an aircraft path. Constraint \eqref{eq:13.1} guarantees the symmetry of swap decisions between flights.

If an aircraft is not swapped with another, we keep the same sequence of flights in an aircraft route as in the original schedule. If the aircraft of flight pair $(i,j)$ is not swapped, a new flight pair $(n,m)$ may be accommodated between flights $i$ and $j$. That is, if the aircraft of flight $j$ is not swapped before its departure, i.e., $\sum_{k \in S(j)} s_{jk}=0$, then constraint \eqref{eq:14.1} ensures that $y_{in}=y_{mj}$ for new flight pair $(n,m)$ as in the constraint \eqref{eq:28} of the first formulation. Otherwise, if the aircraft of flight $j$ is swapped with an aircraft of any flight $k$, i.e., $s_{jk}=1$, then constraint \eqref{eq:15.1} guarantees that $y_{in}=y_{mk}$ for new flight pair $(n,m)$.

Constraint \eqref{eq:4.1} define the arrival time of flights to their destination airport. To model flight departure times and cruise times, we need to make sure that the departure of the successor of flight $i$ is later than the arrival time of flight $i$ plus its aircraft turn time. We first define the successor flight of $i$ in the new schedule. There are three cases for the precedence relations of flight $i$. Case 1: the new flight $n$ follows flight $i$, i.e., $y_{in}=1$, in which case constraint  \eqref{eq:16.1} enables incoming aircraft of flight $i$ to catch new flight $n$. Case 2: \ignore{no new flight after flight $i$ (i.e., $\sum_{n \in N_{O}} y_{in}=0$) and} no swap is made after flight $i$ (i.e., $\sum_{k \in S(j)} s_{jk}=0$), in which case flight $i$ is  followed by  either a new flight or its successor.  If it is followed by a new outbound flight, then the inbound flight of new trip will follow the successor of flight $i$ as well. Therefore, in both situations, departure time of the successor of flight $i$ will be later than the arrival time of flight $i$ plus its turnaround time as it is guaranteed by constraint \eqref{eq:19.1}.  Case 3: \ignore{there is no new flight after flight $i$ and} Aircraft of flight $i$ is swapped with aircraft of flight $k$, in which case flight $k$ follows flight $i$ in the new schedule. Therefore, constraint \eqref{eq:20.1} guarantees the minimum aircraft turn time between flight $i$ and $k$. We ensure the minimum time requirements for the aircraft connections of new flight pair with constraint \eqref{eq:17.1} as well as the aircraft connections between the new flight and its successor with constraint \eqref{eq:18.1}. Similarly, we guarantee the minimum connection time for passengers with constraints \eqref{eq:3.2}. Constraints \eqref{eq:newdeviation} determine the deviation from the original arrival times for existing flights.

The mathematical formulation for the CTC-AS approach also provides aircraft tail assignment. Following the index of decisions $y$ and $x$ determines the sequence of flights operated by the same aircraft. Then, aircraft tail information of the existing flights helps to identify the tail assignment for each route developed. \REV{{The logical constraints \eqref{eq:16.1}-\eqref{eq:20.1} are also represented using the BigM and McCormick inequalities in Section \ref{sec:reformulation}.}}

The mathematical formulations above include nonlinear (convex) fuel and CO$_2$ emission cost and penalty cost terms in the objective function. To efficiently handle the nonlinearity, we use convexification results from mixed-integer conic quadratic optimization. To simplify the presentation, we drop the indices of the variables and parameters as follows:

\begin{equation*}
c(f) =
\begin{cases}
c_o(\alpha  \frac{1}{f}+\beta  \frac{1}{f^2}+\gamma  f^3+\nu  f^2) & \text{if } z=1 \nonumber \\
0 & \text{if }z=0. \nonumber
\end{cases}
\end{equation*}
The function $c\left(f\right)$ with the indicator variable $z$ is discontinuous and its epigraph $E_F=\left\{(x,f,t)\in{\{0,1\} \times \mathbb{R}^2_+} : c(f) \leq t, \ell z \le  f \le u z \right\}$ is nonconvex. The next proposition describes the convex hull of $E_F$. The convexification of convex functions with indicators are discussed in detail in Akt{\"u}rk et al. \cite{akturk2009strong} and G{\"u}nl{\"u}k and Linderoth \cite{gunluk}.

%\todo{$\delta$ is replaced with $\theta$, since it is already used in McCormick ineq.}

\begin{proposition}
\label{conicfuel}
{[\c{S}afak et al. \cite{Safak}]}
The convex hull of the set $E_F$ can be expressed as
\begin{align}
t & \geq c_o(\alpha  p+\beta   q +\gamma   r +\nu  h )\\
{z}^2       & \leq     p f \label{eq:con1}\\
{z}^4       & \leq     {f}^2  q z \label{eq:con2} \\
{f}^4       & \leq     {z}^2 r   f \label{eq:con3}\\
{f}^2       & \leq     h z \label{eq:con4}
\end{align}
 Moreover, each inequality \eqref{eq:con1}--\eqref{eq:con4} can be represented by conic quadratic inequalities.
\end{proposition}

The next proposition \ref{conicdelay} also shows conic quadratic representation of the nonlinear penalty function.

%\todo{I give a new proposition to reformulate the nonlinear penalty for arrival tardiness}

\begin{proposition} The epigraph of penalty function, $E_P=\left\{(b,g) \in{\mathbb{R}^2_+} :  b^{1.5} \leq g \right\}$ is conic quadratic representable.
\label{conicdelay}
\end{proposition}

\begin{proof}
Since  $b^{1.5}$ is a convex function for $b \geq 0$, its epigraph $E_P$ is a convex set. Let us restate $b^{1.5} \leq g$ as

\begin{equation}
b^{4} \leq g^{2} b 1. \label{con1}
\end{equation}

\noindent Observe that \eqref{con1} can be rewritten as two hyperbolic inequalities

\begin{eqnarray}
b^2 \leq x g  \label{eq:con2} \\
x^2 \leq b 1 \label{eq:con3}
\end{eqnarray}

\noindent where $x \geq 0$. According to Ben-Tal and Nemirovski \cite{nemirovski}, hyperbolic inequality \eqref{eq:con2} can be represented by conic quadratic inequality below

\begin{equation}
|| (2b, x-g) || \leq x+g. \nonumber
\end{equation}

\noindent Similarly, hyperbolic inequality \eqref{eq:con3} can be  represented by the conic quadratic inequality. That concludes the proof.

\end{proof}

%\todo{Can we use Latin letters for the variables($p,q,r,h$) and reserve the greeks for data in Proposition 1? Also in section 4.2}

%\begin{proof} Proof is given in \c{S}afak et al. \cite{Safak}.
%\end{proof}

Moreover, the mathematical formulations involve logical constraints. In the next subsection, we replace logical constraints by Big-M constraints. Then, we strengthen the formulation by replacing logical constraints with stronger McCormick inequalities.

\section{Stronger Reformulations}
\label{sec:reformulation}

The nonlinear fuel and emission costs, binary aircraft assignment and swapping decisions and the logical ``if-then" constraints in the models of the preceding section increase the computational burden of solving the problem significantly. In order to solve the problem with less computational effort, in this section we give alternative, stronger reformulations  We first introduce a simple linearizion of the logical constraints using the well known Big-M method with carefully computed constants. Then we improve this formulation using McCormick estimators.

\subsection{Reformulations with Big-M constraints}
\label{Sec:4}

Formulations with logical constraints by means of conditional ``if-then" statements can be numerically more robust than the Big-M formulations if the Big-M formulations use large constants to express the constraints linearly. Solvers may exploit the explicit conditional statements to improve the preprocessing and branching algorithms. Details on logical constraints can be found in \cite{cplex}. However, for our formulations, we are able to carefully tighten the Big-M constants using the implied upper and lower bounds on the variables, leading to more effective formulations. In the following, we will present linear reformulations of the logical constraints of CTC and CTC-AS with the corresponding Big-M constraints.

%\subsection{Big-M reformulation of CRS}

The formulation CTC involves logical constraints \eqref{eq:16}, \eqref{eq:15}, and \eqref{eq:17}. We introduce below three linear constraints to replace these logical constraints, respectively.
%\todo{The bounds are already defined before the models. No need to define here again.}
\begin{proposition}
%\rep{}{Let $d_n^l$ be a lower bound for $d_n$ for all $n \in N_O$ and $d_i^u$ be an upper bound for $d_i$ for all $i \in E_{I}$. Let $u_i^{t(i)}$ be an upper bound of $f_i^{t(i)}$ for all $i \in E_{I}$. Then, f}
For $i \in E_{I}$, $n \in N_O$, inequality %\eqref{eq:bigm1}
\allowdisplaybreaks
\begin{eqnarray}
d_n - d_i -  f_i^{t(i)} \geq \tau_{i}^{t(i)} + \eta_i - \delta^1_{in} \left(1 - y_{in} \right), \label{eq:bigm1}
\end{eqnarray}
where $\delta^1_{in}:=  \max (d_i^u + u_i^{t(i)} +\tau_{i}^{t(i)}+ \eta_i$ $-d_n^\ell, 0)$,
is equivalent to \eqref{eq:16}.
\end{proposition}
%\todo{$M^1_{in}$ is defined differently in the next section. We can't have the same notation meaning two different values. Changing M's to $\delta$'s}
\begin{proof}
For any  $i \in E_{I}$, $n \in N_O$, if $ y_{in} =1$, then constraint \eqref{eq:bigm1} is same as the \eqref{eq:16}. Otherwise, \eqref{eq:bigm1} reduces to the redundant inequality $d_i^u - d_i +   u_i^{t(i)} -  f_i^{t(i)} + d_n - d_n^l \geq 0$ since $u_i^{t(i)}$ is an upper bound for $f_i^{t(i)}$, $d_n^\ell$ is an lower bound for  $d_n$, and $d_i^u$ is an upper bound for $d_i$.
\end{proof}

The rest of the inequalities are stated without proof as they are similar to the one above.
\begin{proposition}%Let  $d_n^u$ be an upper bound for $d_n$ for all $n \in N_I$ and $d_i^l$ be a lower bound for $d_i$ for all $i \in E_{O}$. Let $ u_n^{t}$ be an upper bound of $f_{n}^{t}$ for all $n \in N_I, t \in T$. Then,
For all $ i \in E_{O}, n \in N_I$, inequality %\eqref{eq:bigm2}
\allowdisplaybreaks
\begin{eqnarray}
d_i - d_n - \sum_{t \in T} f_{n}^{t} \geq \tau_{n}^{t(i)} +  \eta_{n}  -  \delta^2_{ni} \left(1 - y_{ni} \right), \label{eq:bigm2}
\end{eqnarray}
where $\delta^2_{ni} := \max \left(d_n^u + \max_{t \in T} u_n^{t}  + \tau_{n}^{t(i)} +  \eta_{n} - d_i^\ell,0\right)$,
is equivalent to \eqref{eq:15}.

\end{proposition}

%\begin{proof} Proof is similar.
%\end{proof}

\begin{proposition} %Let $d_i^u$ be an upper bound for $d_i$ and $d_j^l$ be a lower bound for $d_j$ for all $(i,j) \in C_E$. Let $u_i^{t(i)} $ be an upper bound of $f_i^{t(i)}$ for all $i \in E$. Then,
	For $(i, j) \in C_E$, inequality %\eqref{eq:bigm3}
\allowdisplaybreaks
\begin{eqnarray}
d_{j} - d_i - f_i^{t(i)} \geq \tau_{i}^{t(i)}+ \eta_i - \delta^3_{ij} \left( \sum_{n \in N_O} y_{in} \right), \label{eq:bigm3}
\end{eqnarray}
where $\delta^3_{ij} := \max\left(d_i^u + u_i^{t(i)} + \tau_{i}^{t(i)} + \eta_i - d_j^\ell,0 \right)$,  is equivalent to  \eqref{eq:17}.
\end{proposition}

%\begin{proof} Proof is similar.
%\end{proof}

%\subsubsection{Big-M reformulation of CTC-AS}
%The logical constraints  \eqref{eq:16.1}, \eqref{eq:18.1}, \eqref{eq:19.1}, and \eqref{eq:20.1} of formulation CTC-AS are replaced with the following linear inequalities, respectively.
The logical constraints  \eqref{eq:16.1} - \eqref{eq:20.1}  of formulation CTC-AS are replaced with the following linear inequalities, respectively.

%\todo{Please state these inequalities as propositions as done above for CTC.}
\allowdisplaybreaks

\begin{proposition}
For $i \in E_{I}, n \in N_O$, inequality
\begin{eqnarray}
d_n - d_i - \sum_{t \in T} f_i^t \geq \sum_{t \in T} \tau_{i}^t \ z_i^t + \eta_i - \delta^4_{in}  \left(1 - y_{in} \right),  \label{eq:25.1}
\end{eqnarray}
where  $\delta^4_{in}:= \max\left( d_i^u + \max_{t \in T }u_i^{t} + \max_{t \in T} \tau_{i}^t  + \eta_i -d_n^\ell, 0 \right)$, is equivalent to \eqref{eq:16.1}.
\end{proposition}

\begin{proposition}
For $i \in E_{O}, n \in N_I$, inequality

\begin{eqnarray}
d_i - d_n - \sum_{t \in T} f_{n}^t \geq \sum_{t \in T} \tau_{n}^t  z_{n}^t +  \eta_{n}  -  \delta^5_{ni}  \left(1 - y_{ni} \right), \label{eq:26.1}
\end{eqnarray}
where $\delta^5_{ni}:=\max\left(d_n^u +  \max_{t \in T }u_n^{t} + \max_{t \in T } \tau_{n}^t +  \eta_{n} - d_i^\ell, 0  \right)$, is equivalent to \eqref{eq:18.1}.
\end{proposition}

\ignore{
We restate the constraints \ref{eq:19.1} and \ref{eq:20.1} in Proposition \ref{proof} and Proposition \ref{proof2}, respectively, then we linearize them using the Big-M method.

%\todo{I rewrite the proof.}

\begin{proposition}
	\label{proof}
	For $(i, j) \in C_E$,
	inequalities \eqref{eq:19.1}  can be restated as
	\begin{eqnarray}
	\text{If } \sum_{k \in S(j)} \! s_{jk} = 0, \text{then }  d_i \! + \! \sum_{t \in T} f_i^t \! + \! \eta_i \! +  \! \sum_{t \in T} \tau_{i}^t  z_i^t   \leq d_{j}. \label{eq:nonlinear}
	\end{eqnarray}
\end{proposition}
\begin{proof}
For any $(i,j) \in C_E$, if $\sum_{k \in S(j)} s_{jk}=0$, there are two cases: if $\sum_{n \in N_{O}} y_{in} =0$, then constraint  \eqref{eq:nonlinear} is same as the constraint \eqref{eq:19.1}; otherwise, if $\sum_{n \in N_{O}} y_{in} =1$, then there exists $n^* \in N_O$ such that $y_{in^*}=1$. For feasibility, from constraint \eqref{eq:16.1}, we need $d_i + \sum_{t \in T} f_i^t + \eta_i + \sum_{t \in T} \tau_{i}^t  z_i^t  \leq d_{n^*}$. Similarly, from constraint \eqref{eq:17.1}, we need   $d_{n^*}+ \sum_{t \in T} f_{n^*}^t + \eta_{n^*} + \sum_{t \in T} \tau_{n^*}^t  z_{n*}^t \leq d_{m^*}$ for consecutive new flights $(n^*, m^*)$. Since $\sum_{k \in S(j)} s_{jk} = 0$, constraint \eqref{eq:14.1} implies that $y_{in^*}=y_{m^*j}=1$. Therefore, from constraint \eqref{eq:18.1}, we need $d_{m^*} + \sum_{t \in T} f_{m^*}^t + \eta_{m^*} +\sum_{t \in T} \tau_{m^*}^t  z_{m}^t   \leq d_j$, which implies that $d_i + \sum_{t \in T} f_i^t  + \eta_i +   \sum_{t \in T} \tau_{i}^t  z_i^t   \leq d_{j}$.
%If  $\sum_{k \in S(j)} s_{jk}=1$, then the inequality \eqref{eq:nonlinear} is trivially satisfied.

	%For any $(i, j) \in C_E$, if $\sum_{n \in N_{O}} y_{in} =0$, then constraint  \eqref{eq:19.1}  implies \eqref{eq:nonlinear}. Otherwise, if $\sum_{n \in N_{O}} y_{in} =1$, then there exists $n^* \in N_O$ such that $y_{in^*}=1$. For feasibility, from constraint \eqref{eq:16.1}, we need $d_i + \sum_{t \in T} f_i^t + \eta_i + \sum_{t \in T} \tau_{i}^t  z_i^t  \leq d_n^*$. Similarly, from constraint \eqref{eq:17.1}, we need   $d_{n^*}+ \sum_{t \in T} f_{n^*}^t + \eta_{n^*} + \sum_{t \in T} \tau_{n^*}^t  z_{n*}^t \leq d_{m^*}$ for consecutive new flights $(n^*, m^*)$. Moreover, if $\sum_{k \in S(j)} s_{jk} = 0$, then from constraint \eqref{eq:14.1}, $y_{in^*}=u_{m^*j}=1$. Therefore, from constraint \eqref{eq:18.1}, we need $d_{m^*} + \sum_{t \in T} f_{m^*}^t + \eta_{m^*} +\sum_{t \in T} \tau_{m^*}^t  z_{j}^t   \leq d_j$, which implies that $d_i + \sum_{t \in T} f_i^t  + \eta_i +   \sum_{t \in T} \tau_{i}^t  z_i^t   \leq d_{j}$.
\end{proof}

The following restatement is provided without proof as it is similar to the one above.

\begin{proposition}
\label{proof2}
For $k \in S(j), (i, j) \in C_E$, inequalities \eqref{eq:20.1} can be restated as
\begin{eqnarray}
\text{If } s_{jk}   = 1, \text{then }   d_i \! + \! \sum_{t \in T} f_i^t \! + \! \eta_i \!  + \!  \sum_{t \in T} \tau_{i}^t  z_i^t   \leq  d_{k}. \label{eq:nonlinear2}
\end{eqnarray}
\end{proposition}

Using the Proposition \ref{proof} and Proposition \ref{proof2}, the logical constraints  \eqref{eq:19.1} and  \eqref{eq:20.1} of formulation CTC-AS are replaced with the following linear inequalities, respectively.
}

\begin{proposition}
For $(i, j) \in C_E$, inequality
\begin{eqnarray}
d_{j} - d_i - \sum_{t \in T} f_i^t \geq \sum_{t \in T} \tau_{i}^t  z_i^t + \eta_i - \delta^6_{ij}  \left( \sum_{k \in S(j)} s_{jk} \right), \label{eq:27.11}
\end{eqnarray}
where $ \delta^6_{ij}:= \max\left(d_i^u +  \max_{t \in T }u_i^{t} +  \max_{t \in T } \tau_{i}^t + \eta_i - d_j^\ell, 0 \right)$, is equivalent to \eqref{eq:19.1}.
\end{proposition}

\begin{proposition}
For $(i, j) \in C_E, k \in S(j)$, inequality
\begin{eqnarray}
d_k - d_i -  \sum_{t \in T} f_i^t \geq \sum_{t \in T} \tau_{i}^t  z_i^t + \eta_i - \delta^7_{ik}  \left(1 - s_{jk} \right), \label{eq:28.1}
\end{eqnarray}
where $\delta^7_{ik}:= \max\left(d_i^u +  \max_{t \in T }u_i^{t} +  \max_{t \in T } \tau_{i}^t  + \eta_i - d_k^\ell, 0   \right)$, is equivalent to \eqref{eq:20.1}.
\end{proposition}

%We handle nonlinear fuel and emission cost functions by using mixed integer conic quadratic programming. The results in Section \ref{Computation} indicates that we are able to solve the proposed reformulations within a few minutes for reasonable size problems.

\REV{{We provide the reformulation with Big-M and the hyperbolic inequalities, which can be written as conic quadratic inequalities, below.}}

\allowdisplaybreaks

\begin{eqnarray}
\text{max}&& \sum_{t \in T} \sum_{i \in N} \pi_i  \text{ min} \left(\mu_i, \kappa^t \right)  z_i^t  -\sum_{i \in E} \rho g_i - \phi    \nonumber \\%\tag{$(MICQ+MC)$}\nonumber \\
&& -\sum_{i \in E} \left( \sum_{t \in T} c_o  \left(\alpha_i^t  p_i^t + \beta_i^t  q_i^t + \gamma_i^t  r_i^t + \nu_i^t  h_i^t\right) - c_i^{t(i)}(u_i^{t(i)}) \right)  \nonumber \\
&& - \sum_{t \in T} \sum_{i \in N} c_o  \left(\alpha_i^t   p_i^t + \beta_i^t  q_i^t + \gamma_i^t  r_i^t + \nu_i^t  h_i^t\right) \nonumber \\
&& - \sum_{t \in T} \sum_{i \in E \cup N} \sigma_{i}  \text{ max} \left(0, \mu_i - \kappa^t \right)  z_i^t    - \sum_{i \in E} \sum_{j \in S(i)} \left(\psi/2\right)  s_{ij} \nonumber
\end{eqnarray}
\begin{align}
& \text{s.t.} \quad \text{Const.} \quad \eqref{eq:4}-\eqref{eq:4b},
\eqref{eq:5.1}-\eqref{eq:17.1},
\eqref{eq:3.2}-\eqref{eq:newdeviation1}  && \nonumber \\
& {\left(z_i^t\right)}^2        \leq     p_i^t  f_i^t &&   i \in {E \cup N}, t \in T   \label{eq:34}\\
&{\left(z_i^t\right)}^4        \leq     {\left(f_i^t\right)}^2  q_i^t  z_{i}^t && i \in {E \cup N}, t \in T  \label{eq:35}\\
&{\left(f_i^t\right)}^4        \leq     {\left(z_i^t\right)}^2 r_i^t  f_i^t && i \in {E \cup N}, t \in T  \label{eq:36}\\
&{\left(f_i^t\right)}^2        \leq     h_i^t  z_i^t  &&  i \in {E \cup N}, t \in T  \label{eq:37}\\
&b_{i}^{4} \leq g_{i}^{2} b_{i}  && i \in E \label{eq:conicdelay}\\
&d_n - d_i - \sum_{t \in T} f_i^t \geq \sum_{t \in T} \tau_{i}^t \ z_i^t + \eta_i - \delta^4_{in}  \left(1 - y_{in} \right) && i \in E_{I}, n \in N_O \label{eq:500}\\
&d_i - d_n - \sum_{t \in T} f_{n}^t \geq \sum_{t \in T} \tau_{n}^t  z_{n}^t +  \eta_{n}  -  \delta^5_{ni}  \left(1 - y_{ni} \right) && i \in E_{O}, n \in N_I \label{eq:501}\\
&d_{j} - d_i - \sum_{t \in T} f_i^t \geq \sum_{t \in T} \tau_{i}^t  z_i^t + \eta_i - \delta^6_{ij}  ( \sum_{k \in S(j)} s_{jk} ) &&  (i, j) \in C_E \label{eq:502}\\
&d_k - d_i -  \sum_{t \in T} f_i^t \geq \sum_{t \in T} \tau_{i}^t  z_i^t + \eta_i - \delta^7_{ik}  \left(1 - s_{jk} \right) &&(i, j) \in C_E,  k \in S(j) \label{eq:503}
\end{align}

\REV{{Objective function is slightly different than the original objective function of the proposed model for CTC-AS in Section 3.2. The original objective is represented by the new objective and constraints \eqref{eq:34} - \eqref{eq:conicdelay}, which can be restated by conic quadratic inequalities. Moreover, logical constraints (39)-(42) are represented using the BigM inequalities \eqref{eq:500} - \eqref{eq:503}. The remaining
constraints are same as the original constraints of the proposed model for CTC-AS.}}

As an alternative to the Big-M method, Mannino et. al. \cite{mannino2018hotspot} use the Path\&Cycle Algorithm. In their experiments, the Big-M formulation is slowed down due to the the weak bounds on optimality compared to the Path\&Cycle formulation. As another alternative to the Big-M formulation, we introduce stronger McCormick estimators in the next section.

\subsection{Improved reformulation with McCormick inequalities}
\label{McCormick}
In this section, we improve and strengthen the formulations by using McCormick estimators to represent the logical constraints.
%Although incorporating aircraft swapping option provides a greater flexibility to make a time space for new flights, additional binary assignment and swapping decisions make the
%problem significantly more difficult. We handle the nonlinearity using mixed integer conic quadratic programming and replace logical constraints by Big-M constraints. However, Big-M reformulation of CTC-AS may still require significant compute times to solve instances optimally. Therefore, we provide an improved reformulation using stronger McCormick estimators.
%First, we will restate constraints \eqref{eq:19.1} and \eqref{eq:20.1}, then we will demonstrate the construction of McCormick inequalities only for constraints \eqref{eq:19.1}.
We will demonstrate the construction of McCormick inequalities only for constraints \eqref{eq:19.1}.

%First, we reformulate the Constraints \eqref{eq:19.1} and \eqref{eq:20.1} to replace them with McCormick inequalities.
%\todo{I removed the intermediate if-then inequality to focus on McCormick here only}

Let us define auxiliary variables $v_i = d_i + \sum_{t \in T} f_i^t + \eta_i + \sum_{t \in T} \tau_{i}^t  z_i^t$. %and  $w_{ij} = v_i \left(1 - \sum_{k \in S(j)} s_{jk}  \right)$,
We can state inequality \eqref{eq:19.1} as
\[
d_j \ge v_i \bigg (1 - \sum_{k \in S(j)} s_{jk} \bigg), \ \ (i,j) \in C_E.
\]

%\todo{$w_{ij}$ is introduced as $w_{ij}^1$ in the model. Correct them inside the text}
The problem formulation can be strengthened using linear inequalities based on McCormick estimators for the bilinear terms $w_{ij} = v_i  (1 - \sum_{k \in S(j)} s_{jk}  ),$ $(i,j) \in C_E.$ To do so, note the valid upper and lower bounds on $v_i$:
\begin{eqnarray}
v_i^{u} := d_i^{u} + \max_{t \in T} u_i^{t} + \max_{t \in T} \tau_i^{t} + \eta_i \geq v_i \nonumber\\
v_i^{\ell} := d_i^{\ell} + \min_{t \in T} \ell_i^{t} + \min_{t \in T} \tau_i^{t} + \eta_i \leq v_i. \nonumber
\end{eqnarray}

\noindent Using these bounds on $v_i$, $i \in E$, the following McCormick inequalities \cite{McCormick} are valid for each bilinear term  $\omega_{ij} = v_i \left(1 - \sum_{k \in S(j)} s_{jk} \right):$
\begin{eqnarray}
&&\omega_{ij} \leq  v_i^{u}  \bigg (1 -  \sum_{k \in S(j)} s_{jk}  \bigg)\quad\quad\qquad\qquad(i,j) \in C_E  \label{MC1.1}\\
&&\omega_{ij} \geq  v_i^{\ell}   \bigg (1 - \sum_{k \in S(j)} s_{jk}  \bigg)  \quad\quad\qquad\qquad\, (i,j) \in C_E \\
&&\omega_{ij} \leq  v_i - v_i^{\ell}  \sum_{k \in S(j)} s_{jk}  \quad\quad\quad\qquad\qquad\, (i,j) \in C_E \label{redundant}\\
&&\omega_{ij} \geq  v_i - v_i^{u} \sum_{k \in S(j)} s_{jk}  \quad\quad\quad\qquad\qquad\, (i,j) \in C_E.  \label{MC4.1}
\end{eqnarray}

\noindent Therefore, constraints \eqref{eq:19.1} can be replaced with the following constraints
\begin{eqnarray}
d_j \geq \omega_{ij} \nonumber \\
\eqref{MC1.1} - \eqref{MC4.1}. \nonumber
\end{eqnarray}

Observe that constraints \eqref{redundant} do not need to be added. If $ \sum_{k \in S(j)} s_{jk} =1$,  then constraints \eqref{redundant} and \eqref{MC4.1} become redundant. On the other hand, if $ \sum_{k \in S(j)} s_{jk} =0$, then constraints  \eqref{redundant} and \eqref{MC4.1} ensure that $d_j \geq \omega_{ij}=v_i$. If constraints \eqref{redundant} are not included in the model, then we obtain  $d_j \geq \omega_{ij} \geq v_i$. By this way, we still ensure that flight $j$ cannot depart before the arrival time of flight $i$ plus its turnaround time. Therefore, the optimal solution is same as the original one.

Similarly, logical constraints \eqref{eq:16.1}, \eqref{eq:18.1}, and
\eqref{eq:20.1} can be replaced by the stronger McCormick inequalities. We provide the improved reformulation with the McCormick inequalities and hyperbolic inequalities, which can be written as conic quadratic inequalities, below.

\allowdisplaybreaks
\ignore{
\begin{align}\text{max} \quad & \sum_{t \in T} \sum_{i \in N} p_i \left( min \left(\mu_i, \kappa^t \right) \right)  z_i^t    \\%\tag{$(MICQ+MC)$}\nonumber \\
 & -\sum_{i \in E} \left( \sum_{t \in T} c_o  \left(\alpha_i^t  q_i^t + \beta_i^t  \delta_i^t + \gamma_i^t  \varphi_i^t + \nu_i^t  \vartheta_i^t\right) - c_i^{t(i)}(u_i^{t(i)}) \right)  \nonumber \\
 & - \sum_{t \in T} \sum_{i \in N} c_o  \left(\alpha_i^t   q_i^t + \beta_i^t  \delta_i^t + \gamma_i^t  \varphi_i^t + \nu_i^t  \vartheta_i^t\right) \nonumber \\
& - \sum_{t \in T} \sum_{i \in E \cup N} \sigma_{i}  \left( max \left(0, \mu_i - \kappa^t \right) \right)  z_i^t    - \sum_{i \in E} \sum_{j \in S(i)} \left(\psi  s_{ij}\right)/2  \label{eq:0.1} \\
 \text{subject to} \nonumber
\end{align}
}

\begin{eqnarray}
\text{max}&& \sum_{t \in T} \sum_{i \in N} \pi_i  \text{ min} \left(\mu_i, \kappa^t \right)  z_i^t  -\sum_{i \in E} \rho g_i - \phi    \nonumber \\%\tag{$(MICQ+MC)$}\nonumber \\
&& -\sum_{i \in E} \left( \sum_{t \in T} c_o  \left(\alpha_i^t  p_i^t + \beta_i^t  q_i^t + \gamma_i^t  r_i^t + \nu_i^t  h_i^t\right) - c_i^{t(i)}(u_i^{t(i)}) \right)  \nonumber \\
&& - \sum_{t \in T} \sum_{i \in N} c_o  \left(\alpha_i^t   p_i^t + \beta_i^t  q_i^t + \gamma_i^t  r_i^t + \nu_i^t  h_i^t\right) \nonumber \\
&& - \sum_{t \in T} \sum_{i \in E \cup N} \sigma_{i}  \text{ max} \left(0, \mu_i - \kappa^t \right)  z_i^t    - \sum_{i \in E} \sum_{j \in S(i)} \left(\psi/2\right)  s_{ij} \nonumber
% \label{eq:0.1} \\
\end{eqnarray}
\begin{align}
&\text{s.t. Const.} \quad \eqref{eq:4}-\eqref{eq:4b},
\eqref{eq:5.1}-\eqref{eq:17.1},
\eqref{eq:3.2}-\eqref{eq:newdeviation1}  \nonumber \\
& {\left(z_i^t\right)}^2        \leq     p_i^t  f_i^t &&   i \in {E \cup N}, t \in T   \label{eq:34.1}\\
&{\left(z_i^t\right)}^4        \leq     {\left(f_i^t\right)}^2  q_i^t  z_{i}^t && i \in {E \cup N}, t \in T  \label{eq:35.1}\\
&{\left(f_i^t\right)}^4        \leq     {\left(z_i^t\right)}^2 r_i^t  f_i^t && i \in {E \cup N}, t \in T  \label{eq:36.1}\\
&{\left(f_i^t\right)}^2        \leq     h_i^t  z_i^t  &&  i \in {E \cup N}, t \in T  \label{eq:37.1}\\
&b_{i}^{4} \leq g_{i}^{2} b_{i}  && i \in E \label{eq:conicdelay.1}\\
&d_n \geq \omega_{in}^1  && i \in E_{I}, n \in N_O  \label{MCbas}\\
&\omega_{in}^1 \leq  v_i^{u}  y_{in} && i \in E_{I}, n \in N_O  \label{MC1}\\
&\omega_{in}^1 \geq  v_i^{\ell}   y_{in} && i \in E_{I}, n \in N_O \label{MC1l}\\
&\omega_{in}^1 \leq  v_i - v_i^{\ell}  \left(1 - y_{in} \right) && i \in E_{I}, n \in N_O \label{MC1l.1}\\
&\omega_{in}^1  \geq  v_i - v_i^{u}  \left(1 - y_{in} \right) && i \in E_{I}, n \in N_O  \label{MC4} \\
%&& d_n+ \sum_{t \in T} f_n^t + \eta_n + \sum_{t \in T} \tau_{n}^t  z_n^t \leq d_{m}  \quad\;\;\,\;\,\,\;\;\;\;\;\;\;  (n, m) \in C_N \label{eq:17.2}\\
&d_i \geq \omega_{ni}^2  && i \in E_{I}, n \in N_O  \\
&\omega_{ni}^2 \leq  v_n^{u}  y_{ni} && i \in E_{O}, n \in N_I  \label{MC1}\\
&\omega_{ni}^2 \geq  v_n^{\ell}   y_{ni} && i \in E_{O}, n \in N_I \label{MC2l}\\
&\omega_{ni}^2 \leq  v_n - v_n^{\ell}  \left(1 - y_{ni} \right) && i \in E_{O}, n \in N_I \label{MC2l.1}\\
&\omega_{ni}^2  \geq  v_n - v_n^{u}  \left(1 - y_{ni} \right) && i \in E_{O}, n \in N_I  \label{MC4} \\
& d_j \geq \omega_{ij}^3 &&   (i,j) \in C_E \\
& \omega_{ij}^3 \leq  v_i^{u}  \bigg(1 -  \sum_{k \in S(j)} s_{jk}  \bigg) && (i,j) \in C_E  \label{MC1}\\
&\omega_{ij}^3 \geq  v_i^{\ell}   \bigg (1 - \sum_{k \in S(j)} s_{jk}  \bigg)  && (i,j) \in C_E\label{MC3l}\\
&\omega_{ij}^3 \leq  v_i - v_i^{\ell}  \sum_{k \in S(j)} s_{jk}  && (i,j) \in C_E \label{MC3l.1}\\
&\omega_{ij}^3 \geq  v_i - v_i^{u} \sum_{k \in S(j)} s_{jk} && (i,j) \in C_E  \label{MC4}\\
&d_k \geq \omega_{ik}^4 && k \in S(j), (i,j) \in C_E \\
&\omega_{ik}^4 \leq  v_i^{u}  s_{jk} &&  k \in S(j), (i,j) \in C_E  \label{MC1}\\
&\omega_{ik}^4 \geq  v_i^{\ell}   s_{jk} && k \in S(j), (i,j) \in C_E \label{MC4l}\\
&\omega_{ik}^4 \leq  v_i - v_i^{\ell}   \left(1 - s_{jk}\right)  && k \in S(j), (i,j) \in C_E \label{MC4l.1}\\
&\omega_{ik}^4 \geq  v_i - v_i^{u}  \left(1 - s_{jk} \right) &&  k \in S(j), (i,j) \in C_E  \label{MCson}
\end{align}
%\todo{How about constraints (42)--(49)?}

\section{Computational Study}
\label{Computation}

In this section, we first test and compare the performance of two approaches, CTC and CTC-AS, proposed in the paper in terms of their effectiveness to improve the airline's profitability computationally through a full-factor experimental design. Then we test and compare the effectiveness of the stronger reformulations described in Section~\ref{sec:reformulation} in solving the computationally intensive approach CTC-AS with aircraft swapping.

In the experimental study, we test performance of MICQ  reformulations with Big-M constraints and McCormick inequalities, respectively. All experiments are performed on a workstation with a 3.60 GHz Intel R Xeon R CPU E5-1650 and 32 GB main memory. The mixed-integer conic quadratic reformulations are implemented using JAVA programming language with a connection to IBM ILOG CPLEX Optimization Studio 12.7.1.

We use a sample schedule extracted from the database ``Airline On-Time
Performance Data," provided by the Bureau of Transportation Statistics of
the US Department of Transportation, BTS \cite{BTSperformance}, and query the planned
departure and arrival times of all United Airlines (UA) domestic flights for the date
of 03/01/2018 from the database. In our computations, departure time intervals ([$d^{\ell}, d^{u}$]) are determined by adding/subtracting ninety minutes to planned departure times of the sample schedule. The flight block times in the sample schedule are calculated by taking the difference between the scheduled arrival and departure times. Then, we assume that 30 minutes of a flight block time is non-cruise time ($\nu$) and remaining is cruise time ($u$). In experiments, we compress the cruise times at most 15\%, hence $\ell = 0.85 u$. Moreover, in the proposed models, we keep each aircraft route, i.e., sequence of flights as provided in the sample schedule.

In this study, new flights are connected to the existing schedule at hub airport ORD. Therefore, in the sample schedule, we remove the route  of aircraft that does not visit ORD airport for this particular day. This is reasonable, since the flights of aircraft which does not visit ORD will not be affected by introducing new flights. The resulting sample schedule includes 300 flights operated by 81 aircraft.

\subsection{Description of the data for the experimental study}

In order to analyze the effects of problem
parameters on the airline's profit, we conduct a $2^k$ full-factorial
experimental design. The experimental factors and their
levels are given in Table \ref{factor}.

\setcounter{table}{2}
\begin{table}[ht]\footnotesize
	\centering
	\caption{Factor values. }
	\label{factor}%
	\begin{tabular}{ r| c| c }
		\hline \hline %\cline{2-3}
		\multicolumn{1}{c}{ }&\multicolumn{2}{|c}{Levels}\\
		\hline
		\multicolumn{1}{c|}{Factor Description} & \multicolumn{1}{c|}{Low} & \multicolumn{1}{c}{High} \\
		\hline
		$c_{fuel}$ (\$/kg) &  0.6 &  1.2 \\
		$\sigma_b$ (\$/passenger) & 60 &  200 \\
		$\psi$ (\$/swap)  & 500 & 1000 \\
		\hline \hline
	\end{tabular}
	
\end{table}

The fuel prices for lower and higher settings, respectively, are estimates based on the history of fuel prices obtained from IATA fuel price monitor \cite{fuelprice}, which shows a fluctuation between \$0.6/kg and \$1.2/kg during years 2008 -- 2018.
In the table $\sigma_b$ is a base value for the opportunity cost for each of the spilled passengers due to the insufficient seat capacity of the
aircraft. \c{S}afak et al. \cite{Safak} express the cost of spilled passenger for each flight using airport congestion coefficients, e.g., favoring the populated markets as follows:
\begin{equation}
\sigma_i = \sigma_b {\left(e_{O_i}\right)} {\left(e_{D_i}\right)},\quad i\in E \cup N,
\label{eq:spilledcost}
\end{equation}
\noindent where $e_{O_i}$ and $e_{D_i}$ represent the congestion coefficients for the origin and destination airports of flight $i \in E \cup N$. These coefficient values are provided in \c{S}afak et al. \cite{Safak}.
$\psi$ is the cost of changes caused by swapping the aircraft of the flights. For low and high values of swap cost, we have used \$500 (proposed by Marla et al. \cite{Marla}) and \$1000, respectively.

We consider six aircraft types and list the fuel burn related parameters, the corresponding maximum range cruise (MRC) speed, and the seat capacity  in Table \ref{aircraftParam}.
The coefficients of the fuel burn function \eqref{eq:cost}, $\alpha_i^t, \beta_i^t, \gamma_i^t, \nu_i^t$, are calculated as specified in \c{S}afak et al. \cite{Safak} using the corresponding values of fuel burn related parameters in Table \ref{aircraftParam}.

\begin{table}[ht]\scriptsize
	\centering
	\caption{Aircraft parameters. }
	\label{aircraftParam}%
	\begin{tabular}{r| c c c c c c}
		\hline \hline
		Aircraft type       & B727 228      & B737 500  &MD 83 & A320 111   & A320 212      & B767 300\\
		\hline
		Capacity            &134               & 122         &148            & 172       & 180           & 218   \\
		Mass (kgs)          & 74000        & 50000          &61200          & 62000     & 64000         & 135000\\
		Surface($m^2$)        & 157.9            & 105.4      &118             & 122.4     & 122.6         & 283.3\\
		$C_{D0,CR}$           & 0.018          & 0.018        &0.0211            & 0.024     & 0.024         & 0.021\\
		$C_{D2,CR}$           & 0.06         & 0.055          &0.0468          & 0.0375    & 0.0375        & 0.049\\
		$Cf_1$                &0.53178        & 0.46         &0.7462            & 0.94      & 0.94          & 0.763\\
		$Cf_2$                &276.72          & 300           &638.59        & 50000     & 100000        & 1430\\
		$cf_{CR}$            &0.954                & 1.079    &0.9505         & 1.095     & 1.06          & 1.0347\\
		MRC speed (km/h)           & 867.6       & 859.2            &867.6         & 855.15    & 868.79        & 876.70 \\
		$\tau_{b}^t$ (min)           & 32         &36              & 26             & 28        & 30            & 40   \\
		\hline \hline
	\end{tabular}
	
\end{table}

For a flight $i$,
the aircraft turnaround time ($\tau_{i}^{t}$) needed to prepare the aircraft for the next flight is estimated using the expression
\begin{equation}
\tau_{i}^t = \tau_{b}^t \cdot{e_{D_i}}\quad t \in T,
\label{tab:turntime}
\end{equation}
where $\tau_{b}^t$ is a base value for aircraft turnaround time.  Therefore, turnaround time of an aircraft visiting a congested airport will take longer. The calculated aircraft turnaround times match with the aircraft turnaround times given in Ar{\i}kan et al. \cite{arikanDes}.

The passenger demand for existing flights are generated uniformly between 110 and
134, 110 and 122, 110 and 148, 150 and 172, 160 and 180, 160 and
218, for aircraft types B727 228, B737 500, MD 83, A320
111, A320 212 and B767 300, respectively. Without loss of generality, we assume that the original
aircraft assignments meet all passenger demand. Under this experimental setting, we can analyze the performance of aircraft
swapping while trading-off between the cost of fuel burn and cost of spilled passengers.

\subsection{Performance analysis of CTC and CTC-AS} While aircraft swapping in addition to re-timing departures and cruise speed control provides a greater flexibility to accommodate the new flights, it is of interest to study the incremental increase in the airlines profit due to the heavy computational burden of solving CTC-AS.

In this experimental study, we use  300 flights operated by 81 aircraft of the sample schedule. We consider adding round-way trips \{ORD-IAH\}, \{ORD-BOS\} and \{ORD-MSP\} so that there are six new flights to be added into the new schedule. The demand for the new flights are generated uniformly between 120 and 200.  The estimated demand ranges are defined based on the seat capacity of the aircraft types commonly used to operate these trips. The fares for the new flights are generated uniformly between USD 120 and USD 350, based on an analysis of the ticket prices of the flights in these trips for United Airlines. According to Eurocontrol's report \cite{eurocrew} on dynamic cost indexing, total crew cost per block hour varies between \$280 and \$800. In this study, we let the crew cost for each new flight as \$400/hr. Using an expected block time for each new flight,  the total crew cost for six new flights is calculated as \$6500. For each new flight, we also consider \$1500 service cost including landing fees, ground handling fees, insurance fees, onboard services costs, etc.

We design a $2^3$ experimental study with two levels for each experimental
factor. For each combination of the factor levels, we solve 10 randomly generated instances with
the approaches CTC and CTC-AS, respectively. In both formulations, we replace logical constraints with McCormick inequalities and handle the nonlinear fuel and emission costs using mixed integer conic quadratic programming described in Proposition \ref{conicfuel}. Then, each instance is approximately solved less than \REV{\textbf{sixty seconds}} by the CTC approach. On the other hand, with formulation CTC-AS, the average CPU time is 6,100 seconds due to increased number of conic constraints with binary assignment decisions
and binary swapping decisions. Despite the additional computational complexity,
the CTC-AS approach provides substantial profit improvement over
the simpler cruise time controllability approach CTC, calculated as

\[
\text{Profit improvement ($\%$)} = 100 \times \frac{\text{Optval (CTC-AS) -- Optval (CTC)}} {\text{Optval (CTC)}} \cdot
\]

Table  \ref{comparison} summarizes the results for 80 instances. Observe that the profit improvement significantly increases as the fuel price increases. As the fuel expenses are a major cost component of airlines, this cost component becomes more important with higher fuel price. In order to reduce fuel burn, CTC-AS has an advantage of reassigning flights to more fuel efficient aircraft.   For high fuel price, the profit improvement can reach to 131\% over CTC. On the other hand, if the spill cost is high, then the profit improvement decreases. Similarly,  the profit improvement decreases as the swapping cost parameter increases. Because, swapping the aircraft of a flight with a smaller aircraft may spill some of the passengers, and such swaps incur additional spilled passenger cost for CTC-AS approach. CTC-AS approximately yields a 53\% improvement in airline's profit compared to CTC for the factor values analyzed in this study.

%\todo{a new table}

\begin{table}[ht]\footnotesize
	\centering
	\caption{Profit improvement of CTC-AS over CTC. }
	\label{comparison}
	\begin{tabular}{ l l  r  r  r}
		\hline
		\hline
		\multicolumn{2}{c}{}& \multicolumn{3}{c}{Profit Impr. (\%) }  \\
		\multicolumn{2}{c}{}&\multicolumn{1}{c}{min}&\multicolumn{1}{c}{avg}&\multicolumn{1}{c}{max} \\ \hline
		$c_{fuel}$& Low & 8		&25		&44 \\
		&High	&42		&81		&131 \\
		$\sigma_b$& Low	&25		&66		&131 \\
		&High&8		&39		&107 \\
		$\psi$	&Low 	&10		&55		&131\\
		&High	&8		&51		&125\\
		\hline
		& avg 	&	&53	& \\
		\hline
		\hline
	\end{tabular}
	
\end{table}

%\todo{I think the what-if analysis is still part of the performance analysis of the approaches. Changed it to a subsection}
\subsubsection{What-if analysis on the number of aircraft swaps} CTC-AS approach utilizes the aircraft swapping mechanism to reduce the fuel burn. On the other hand, if the aircraft of a flight is swapped with a smaller aircraft, then some of the passengers of the subsequent flights might be spilled. Therefore, CTC-AS approach trades-off the cost of fuel burn with the cost of spilled passengers. To see the effect of the number of swaps, we restrict the maximum number of swaps with the following constraint:
\begin{equation}
\sum_{i \in J} \sum_{j \in AS(i)} x_{ij} \leq max\_swap
\end{equation}

\noindent %where $nb_{swap}$ is the maximum number of swaps determined by an airline.
A schedule planner can specify and modify the maximum number of swaps and analyze the influence of the number of swaps on the airline's profit. In Figure \ref{fig:whatif1}, we provide the efficient frontier for a problem instance with 300 flights and 81 aircraft solved with different levels of factors in Table \ref{level}. If the fuel price is high, then the airline's profit significantly increases as the number of swaps increases. Since the fuel cost is the major cost component of airlines, fuel burn
has a greater influence on airline's profit in this case. The total fuel burn can be reduced by reassigning the fuel efficient aircraft to longer trips and using the less efficient aircraft for shorter trips. \ignore{This reassignment can be achieved by swapping the aircraft among flights. } On the other hand, if the fuel price is low and the spill parameter is high, the profit is slightly improved as the number of swaps increases and the profit remains constant after seven swaps. In this case, it is not preferred to spill passengers due to the high spill cost. As shown in Figure \ref{fig:whatif2}, the percentage of spilled passengers is the lowest for factor combinations 3 and 7 as expected. The diminishing rate of return in the profit increase allows the airline to limit the number of swaps to a few to gain a large benefit with a small number of swaps.

\begin{table}[ht]\footnotesize
	\centering
	\caption{Combination of factor levels.}
	\label{level}
	\begin{tabular}{ c r r r}
		\hline
		\hline
		\multicolumn{1}{c}{ } & \multicolumn{3}{c}{Levels of}\\
		Factor combination	& \multicolumn{1}{c}{$\psi$ }& \multicolumn{1}{c}{$\sigma_b$}& \multicolumn{1}{c}{$c_{fuel}$}\\
		\hline
		1& Low & Low & Low \\
		2 & Low & Low & High \\
		3& Low & High & Low \\
		4 & Low & High & High \\
		5& High & Low & Low \\
		6 & High & Low & High \\
		7& High & High & Low \\
		8 & High & High & High \\
		\hline
		\hline
	\end{tabular}
	
\end{table}

\begin{figure}[htbp]
	\centering
	\includegraphics[width=0.75\textwidth]{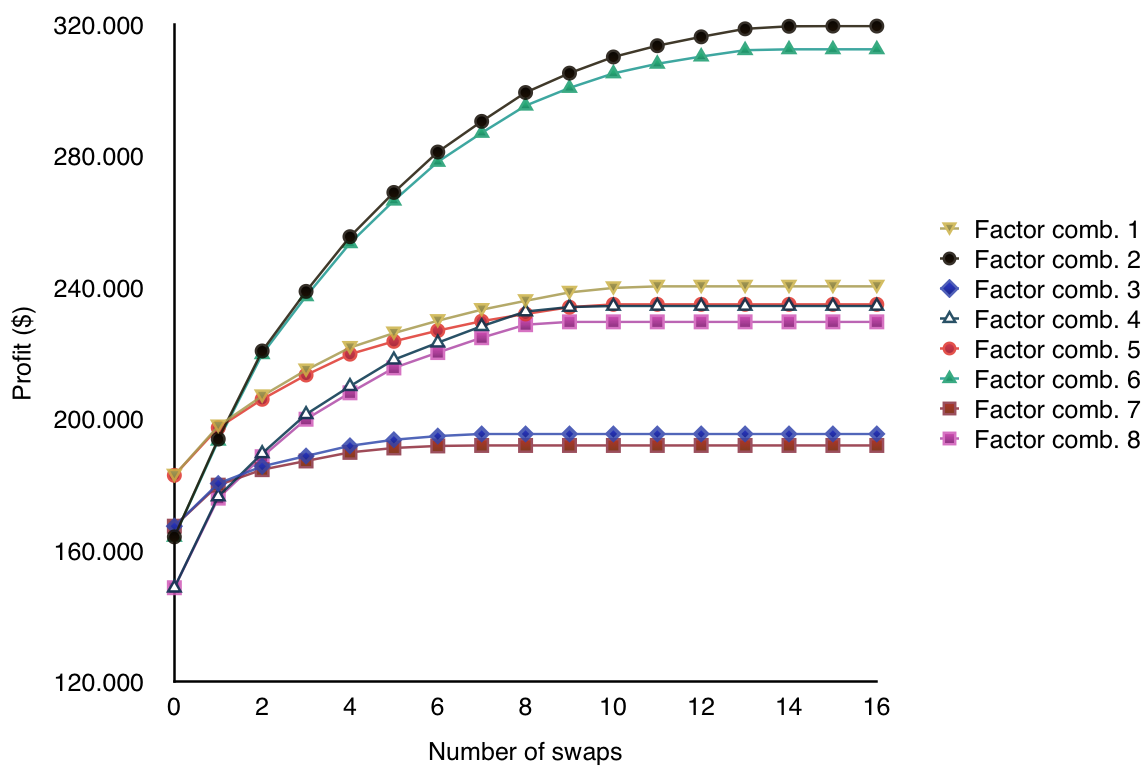}
	\caption{Efficient frontier of aircraft swaps.} \label{fig:whatif1}
\end{figure}

\begin{figure}[htbp]
	\centering
	\includegraphics[width=0.75\textwidth]{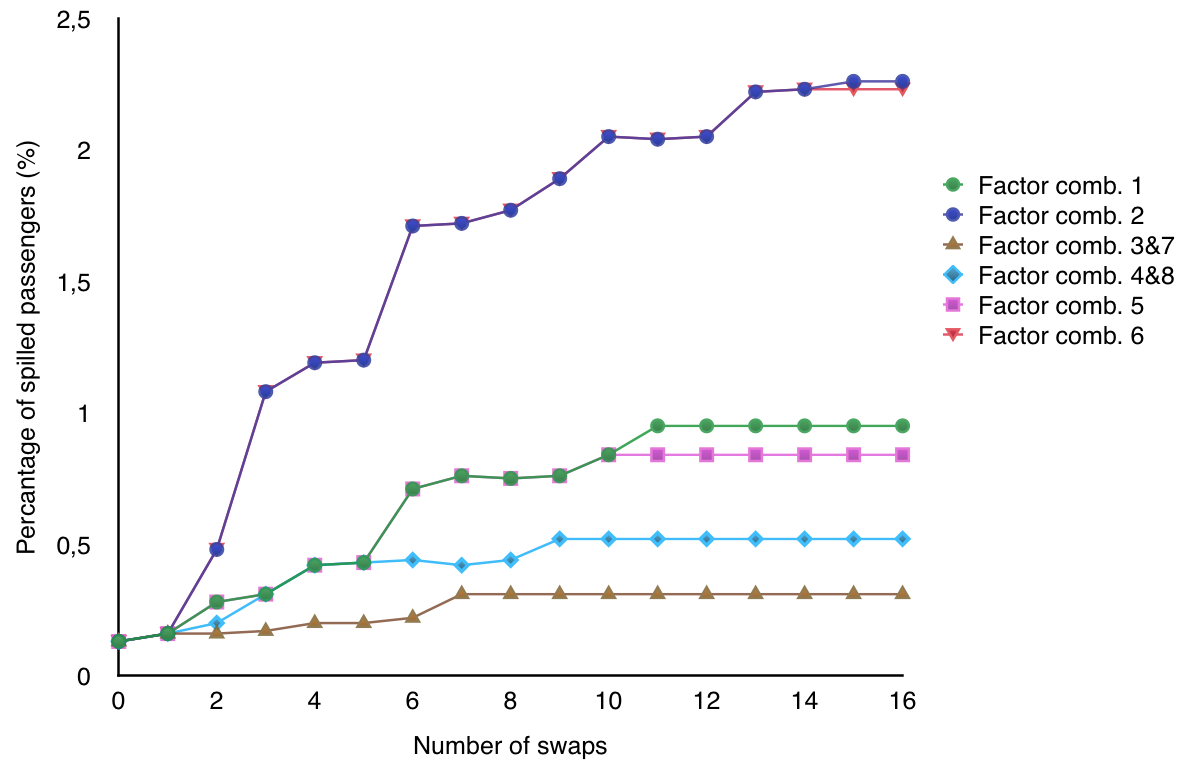}
	\caption{The effect of swap decisions on spilled passengers.} \label{fig:whatif2}
\end{figure}

The computational results clearly demonstrate that the benefit of CTC-AS over CTC and highlights the need to address the computational difficulty for solving the CTC-AS  model. In the next section, we will test the performance of reformulations of CTC-AS.

\subsubsection{What-if analysis for aircraft dependent swap costs.} An airline may consider its passenger's satisfaction when the aircraft of a passenger's flight is swapped. While passengers would be satisfied when the aircraft of their flights is swapped with a larger one, in other case, passengers would not feel comfortable with smaller aircraft. Therefore, in this section, we consider the aircraft dependent swapping cost parameters provided in Table \ref{swapDependent}. If aircraft type ($t$) in the first row is swapped with another aircraft type ($t^{'}$) in one of the columns, we incur an additional penalty cost $\Phi^{ti t^{'}}$ to original swap cost $\psi$ for this swap.

\begin{table}[ht]\scriptsize
	\centering
	\caption{Aircraft dependent swap cost. }
	\label{swapDependent}%
	\begin{tabular}{r| c c c c c c}
		\hline \hline
		Aircraft type       & B727 228      & B737 500  &MD 83 & A320 111   & A320 212      & B767 300\\
		\hline
		B727 228		& 0			&	0		&0	&0		&0		&0\\
		B737 500		&0			& 0			&0		&0		&0		&0\\
		MD 83		&100			&100			&0		&0		&0		&0\\
		A320 111		&150			&150			&100		&0		&0		&0\\
		A320 212		&150			&150			&100		&0		&0		&0\\
		B767 300		&200			&200			&150		&100		&100		&0\\
		\hline \hline
	\end{tabular}
	
\end{table}

We redefine the cost of swap ($\phi_{ij}$) between the aircraft of flights $i$ and $j$ and calculate as:

\begin{equation}
\phi_{ij}  = \psi + \Phi^{t(i), t(j)}. \nonumber
\end{equation}

\noindent $\psi$ is still considered as \$500 and \$1000 for low and high values, respectively. For each combination of fuel price, base spill cost and swap cost levels, we solve ten instances. Then, we report the minimum, average and maximum profit improvements of CTC-AS approach over the CTC approach in Table \ref{profitSwap}.

\begin{table}[ht]\footnotesize
	\centering
	\caption{Profit improvement of CTC-AS over CTC. }
	\label{profitSwap}
	\begin{tabular}{ l l  r  r  r}
		\hline
		\hline
		\multicolumn{2}{c}{}& \multicolumn{3}{c}{Profit Impr. (\%) }  \\
		\multicolumn{2}{c}{}&\multicolumn{1}{c}{min}&\multicolumn{1}{c}{avg}&\multicolumn{1}{c}{max} \\ \hline
		$c_{fuel}$& Low & 8		&24		&43 \\
		&High	&40		&78		&127 \\
		$\sigma_b$& Low	&24		&64		&127 \\
		&High&8		&38		&103 \\
		$\psi$	&Low 	&10		&53		&127\\
		&High	&8		&49		&121\\
		\hline
		& avg 	&	&51	& \\
		\hline
		\hline
	\end{tabular}
	
\end{table}

As expected, if there is an additional aircraft dependent swap cost, then the profit improvement slightly decreases compared to the single swap cost case provided in Table 5. On the other hand, CTC-AS approach with aircraft type dependent swap cost still significantly increases the profit  compared to the CTC approach, if the fuel price increases. Because, CTC-AS approach has an advantage of reducing the fuel burn by reassigning the fuel efficient aircraft among subset of flights. However, Table \ref{percentageSpill} shows that these reassignments spill more passengers as the fuel price increases. If the spill cost parameter increases, then the number of swaps decreases so that the number of spilled passengers decreases. The results indicate that an average of 0.9\% of passengers are spilled while maximizing the airline profit to capture the additional demand of new flights.

\begin{table}[ht]\footnotesize
	\centering
	\caption{Percentage of the spilled passengers}
	\label{percentageSpill}
	\begin{tabular}{ l l  r  r  r}
		\hline
		\hline
		\multicolumn{2}{c}{}& \multicolumn{3}{c}{Percentage of }  \\
		\multicolumn{2}{c}{}& \multicolumn{3}{c}{spilled pax. (\%) }  \\
		\multicolumn{2}{c}{}&\multicolumn{1}{c}{min}&\multicolumn{1}{c}{avg}&\multicolumn{1}{c}{max} \\ \hline
		$c_{fuel}$& Low & 0.2		&0.6		&1.3 \\
		&High	&0.3		&1.2		&2.3 \\
		$\sigma_b$& Low	&0.7		&1.4		&2.3 \\
		&High&0.2		&0.4		&0.7 \\
		$\psi$	&Low 	&0.2		&0.9		&2.3\\
		&High	&0.2		&0.9		&2.3\\
		\hline
		& avg 	&	&0.9	& \\
		\hline
		\hline
	\end{tabular}
	
\end{table}

\subsection{Analysis of the reformulations}
The original formulation in Section \ref{SecondFormulation} has nonlinear fuel burn and carbon emission functions in the objective. To efficiently handle them, we propose a mixed-integer conic quadratic reformulation. Here, we compare two alternative mixed-integer conic quadratic reformulations referred to as MICQ1 and MICQ2. The second one MICQ2 is formulated using the strengthened inequalities as shown in Table \ref{alternative}. It is clear that inequalities of MICQ2 are valid for MICQ1 since for $z \in \{0,1\}$ they reduce to inequalities of MICQ1. Recall that the additional constraints $\ell z \leq f \leq u z$, which force $f$ to zero whenever $z$ is zero.

With each of these formulations we compare two different ways of forcing the logical constraints. First, ``MICQ1/2+BigM", corresponds to the mixed integer conic quadratic reformulation, where the logical constraints are replaced with the Big-M constraints \eqref{eq:500}-\eqref{eq:503}. Second, ``MICQ1/2+MC", replaces the logical constraints with the McCormick inequalities \eqref{MCbas}--\eqref{MCson} described in Section \ref{McCormick}. We do not need to add inequalities \eqref{MC1l.1}, \eqref{MC2l.1}, \eqref{MC3l.1} and \eqref{MC4l.1} as discussed in Section \ref{McCormick}, therefore we do not include them in the experiments.

We perform a $2^3$ experimental design with two levels of experimental factors in Table \ref{factor}. For each of these eight factor combinations, we generate 10 instances, resulting in a total of 80 instances.  A time limit of 18000 seconds is used for each runtime of each instance.

\begin{table}[ht]\footnotesize
	\centering
	\caption{Alternative conic formulations.}
	\label{alternative}
	\begin{tabular}{ c l  l  }
		\hline
		\hline
		\multicolumn{1}{c}{}& \multicolumn{1}{l}{MICQ1} & \multicolumn{1}{l}{MICQ2} \\
		\hline
		& \(\displaystyle z^2 \leq p \cdot f	\)	&\(\displaystyle z^2 \leq p \cdot f	\) \\
		Hyperbolic 	& \(\displaystyle z^4 \leq f^2 \cdot  q \cdot 1  \)	& \(\displaystyle z^4 \leq f^2 \cdot  q \cdot z  \) \\
		inequalities 	& \(\displaystyle {f}^4        \leq     {1}^2 \cdot r \cdot f \)	& \(\displaystyle {f}^4        \leq     {z}^2 \cdot r \cdot f \) \\
		&\(\displaystyle {f}^2        \leq     h \cdot 1  \)	& \(\displaystyle {f}^2        \leq    h \cdot z  \) \\
		\hline
		\hline
	\end{tabular}
\end{table}

The computational performance results with the alternative formulations are summarized in Tables~\ref{comparison2}--\ref{comparison3}. Table \ref{comparison2} displays the results with the conic formulation MICQ1, whereas Table \ref{comparison3} presents the improved results with the strengthened conic formulation MICQ2. For each fuel price, base spill cost, and swap cost parameters, the first column ``\# nodes" reports the average number of nodes explored in the branch-and-bound algorithm. The second column ``time" reports the average CPU time (in seconds) for the instances that could be solved  to optimality within the time limit with the number of such instances in the parenthesis. The symbol (-) indicates that none of the instances could be solved to optimality within the time limit. The third column ``gap" reports the average percentage optimality gap between the best bound at termination and the  integer objective  with the number of such instances that could not be solved to optimality in parenthesis.

Table \ref{comparison2} shows that none of the instances could be solved to optimality
using MICQ1 with either Big-M or McCormick inequalities within the time
limit. On the other hand, the computational performances of
strong conic formulations MICQ2+BigM/MC are significantly better. The formulation MICQ2+BigM solves the most of instances to optimality within the time limit as indicated in Table \ref{comparison3}. Even, for each instance that could not be solved by formulation MICQ1+BigM  to optimality, the stronger formulation MICQ2+BigM achieves a lower optimality gap. When the McCormick valid inequalities are added to the strong conic formulation MICQ2, the computational performance further improves. The McCormick inequalities help to solve all instances
faster within the time limit. For some instances, the strong conic programming
formulation MICQ2 with McCormick estimators is solved more than
twice as fast than the best alternative.

\begin{table}[ht]\footnotesize
	\setlength{\tabcolsep}{1pt}
	\centering
	\caption{Comparison of conic formulations.}
	\label{comparison2}
		\scalebox{0.9}{
	\begin{tabular}{ r r r | r c  r| r c r }
		\hline
		\hline
		\multicolumn{3}{c}{} &\multicolumn{3}{c}{MICQ1+Big-M} &\multicolumn{3}{c}{MICQ1+MC}  \\
		\multicolumn{1}{c}{$\psi$}& \multicolumn{1}{c}{$\sigma_b$}&  \multicolumn{1}{c}{$c_{fuel}$ }&\multicolumn{1}{c}{\# nodes}&  \multicolumn{1}{c}{cpu (sec)}& \multicolumn{1}{c}{gap (\%)}& \multicolumn{1}{c}{\# nodes}&  \multicolumn{1}{c}{cpu (sec)} & \multicolumn{1}{c}{gap (\%)} \\
		\hline
		500	&60		&0.6		&55726	&-		&11.69(10)	&35416	&-		&7.78(10)		\\
		&		&1.2		&51729	&-		&13.71(10)	&33419	&-		&9.38(10)		\\
		&200		&0.6		&101801	&-		&7.09(10)		&62260	&-		&4.83(10)			\\
		&		&1.2		&70944	&-		&8.55(10)		&39552	&-		&4.74(10)			\\ \hline
		1000		&60		&0.6		&56412	&-		&11.99(10)		&36445	&-		&7.68(10)	\\
		&		&1.2		&56506	&-		&13.45(10)	&35584	&-		&9.01(10)			\\
		&200		&0.6		&102414	&-		&7.00(10)		&55232	&-		&4.32(10)			\\
		&		&1.2		&71287	&-		&7.74(10)		&37564	&-		&5.25(10)		\\
		\hline
		\hline
	\end{tabular}
}
\end{table}

\begin{table}[ht]\footnotesize
	\setlength{\tabcolsep}{1pt}
	\centering
	\caption{Comparison of the strengthened conic formulations.}
	\label{comparison3}
		\scalebox{0.9}{
	\begin{tabular}{ r r r |  r r  r| r r r }
		\hline
		\hline
		\multicolumn{3}{c}{} &\multicolumn{3}{c}{MICQ2+Big-M} &\multicolumn{3}{c}{MICQ2+MC}  \\
		\multicolumn{1}{c}{$\psi$}& \multicolumn{1}{c}{$\sigma_b$}&  \multicolumn{1}{c}{$c_{fuel}$ }&\multicolumn{1}{c}{\# nodes}&  \multicolumn{1}{c}{cpu (sec)}& \multicolumn{1}{c}{gap (\%)}& \multicolumn{1}{c}{\# nodes}&  \multicolumn{1}{c}{cpu (sec)} & \multicolumn{1}{c}{gap (\%)} \\
		\hline
		500	&60		&0.6		&13511	&12603(9)		&0.07(1)	&3538	&8206(10)		&0(0)		\\
		&		&1.2		&25209	&12988(6)		&1.46(4)	&6677	&8688(10)		&0(0)		\\
		&200		&0.6		&19908	&5798(9)		&3.48(1)	&7640	&3436(10)		&0(0)			\\
		&		&1.2		&12981	&7304(9)		&2.35(1)	&3783	&3033(10)		&0(0)		\\ \hline
		1000		&60		&0.6		&9900	&10601(9)		&1.00(1)	&3641	&8074(10)		&0(0)	\\
		&		&1.2		&17262	&12930(8)		&0.47(2)	&6161	&8803(10)		&0(0)			\\
		&200		&0.6		&13969	&5085(9)		&3.06(1)	&6560	&3380(10)		&0(0)			\\
		&		&1.2		&10958	&7043(9)		&0.08(1)	&8935	&5608(10)		&0(0)		\\
		\hline
		\hline
	\end{tabular}
}

\end{table}

In Table \ref{avgPerforms}, we summarize the results for 80 instances. The conic formulation MICQ1 including BigM inequalities could not solve any instances to optimality within the time limit. If we replace BigM inequalities with McCormick inequalities in formulation MICQ1, then the average optimality gap is decreased from 10.15\% to 6.61\%. If we use strong conic formulation MICQ2 with BigM inequalities, we can further decrease the average optimality gap to 1.40\% over 12 instances that are not solved to optimality. When we reformulate the logical constraints with McCormick inequalities in strong conic formulation MICQ2, the computational performance is the best. All instances are solved to optimality with a dramatic reduction in number
of nodes explored.

\begin{table}[ht]\footnotesize
	\centering
	\caption{Average performances of conic formulations.}
	\label{avgPerforms}
	\begin{tabular}{ r| r c r}
		\hline
		\hline
		\multicolumn{1}{c}{ } & \multicolumn{1}{c}{\# nodes} & \multicolumn{1}{c}{\# CPU} & \multicolumn{1}{c}{\# Gap}\\
		\hline
		MICQ2+MC& 5867 & 6153(80) & 0(0) \\
		MICQ2+Big-M & 15462 & 9294(68) & 1.40(12) \\
		MICQ1+MC& 41934 & - & 6.61(80) \\
		MICQ1+Big-M & 70852 & - & 10.15(80) \\
		\hline
		\hline
	\end{tabular}
	
\end{table}

%\todo{I added a new discussion and a figure for the analysis of number of nodes}
Figure \ref{fig:nodes} analyzes the average number of nodes explored in the branch-and-bound algorithm for each reformulation. For each factor combination, conic formulation MICQ1 with Big-M constraints explores a large number of nodes within the time limit.  By adding McCormick inequalities to MICQ1, the average number of nodes is reduced by half, but it is still more than three times required by the conic formulations MICQ2. With strengthened inequalities of conic formulation MICQ2+Big-M, a large number of the nodes are fathomed and the optimal solutions are found for many instances within the time limit. Further, conic formulation MICQ2 can be improved and strengthened by adding McCormick inequalities, which leads to further significant reduction in the number of nodes.

\begin{figure}[htbp]
\centering
\includegraphics[width=0.7\textwidth]{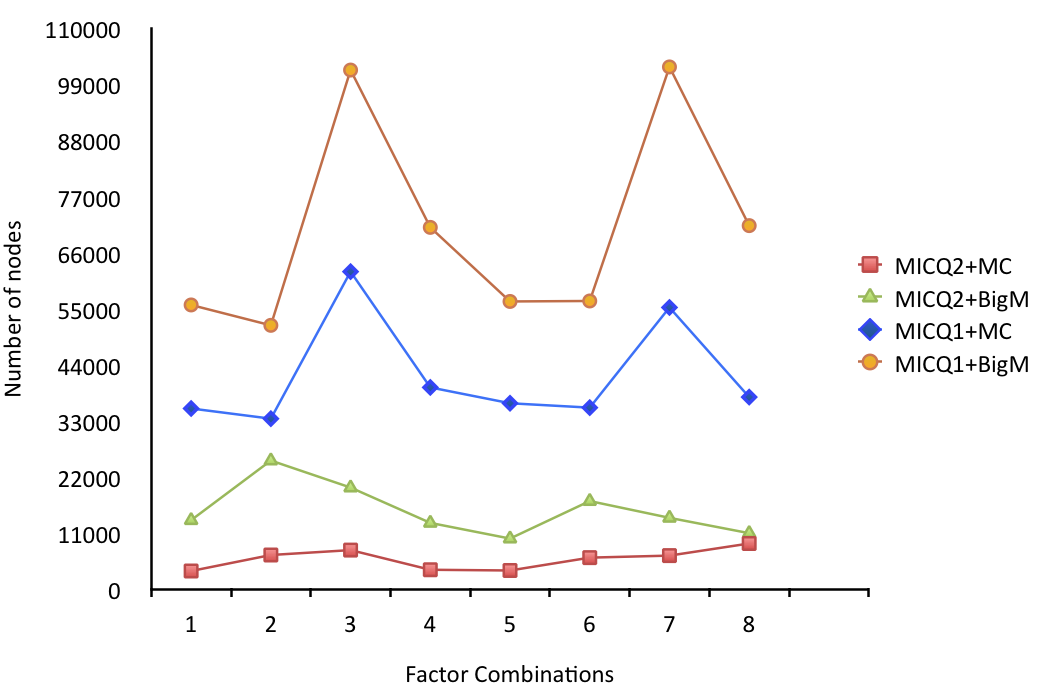}
\caption{Analysis on number of nodes.} \label{fig:nodes}
\end{figure}

\section{Conclusion}
\label{Sec:6}
We propose two approaches to accommodate new flights into an existing flight schedule of a particular day. Both of the approaches make use of re-timing of the flight departure times and cruise time controllability to reduce the block times of the existing flights, thereby making time to operate the new flights in the schedule. The second approach additionally takes the advantage of flexibility offered by aircraft swapping among flights.  The second approach provides substantial cost savings in fuel burn by reassigning flights to fuel-efficient aircraft. However, the nonlinear fuel and emission costs together with additional binary swapping and assignment decisions significantly complicate the problem. To overcome the computational difficulty, we present strong conic quadratic reformulations. The experiments show the superiority of strong conic formulations over an alternative conic formulation. The alternative conic formulation times out for all test instances. On the other hand, the Big-M reformulation of logical constraints with strengthened conic inequalities can solve most of test instances to optimality. For only 12 instances over 80, Big-M reformulation provides an average of 1.40\% optimality gap at termination. As an alternative to Big-M method, when we add McCormick inequalities to the strong conic formulations, all test instances can be solved to optimality with a dramatic reduction in the number of nodes.

In conclusion, we provide two alternative approaches as CTC and CTC-AS with different quality of results and computational difficulties  to an airline. While CTC-AS provides an average of 53\% profit improvement over CTC, the average CPU time to solve the CTC-AS to optimality by a strong conic reformulation together with McCormick inequalities is 6,100 seconds. On the other hand, CTC approach can be easily solved within \REV{\textbf{sixty seconds}}.

This study may lead to several potential research directions.
The computational advantages of the strong conic quadratic models may pave the way for researchers to integrate the consideration of crew itineraries while determining the flight departures. Another extension of this study would
be addressing a strategic planning problem that aims at introducing new flights to new demand points for the next season. Many potential demand scenarios considering the competitor's flights could be analyzed. Moreover, leasing an aircraft to hedge for the demand uncertainties may be an additional
mechanism to introduce new flights.

\section{Acknowledgement}

The authors thank the editor and three anonymous referees for their constructive comments and suggestions that significantly improved this paper. This research was conducted while \"{O}zge \c{S}afak was a visiting scholar at the University of California-Berkeley with support from the Scientific and Technological Research Council of Turkey under Grant 2214. In addition, this research was partially supported by T{\"U}BITAK [Grant 116M542]. Alper Atamt\"urk was supported, in part, by grant FA9550-10-1-0168 from the Office of the Assistant Secretary of Defense for Research and Engineering.

%\bibliography{ref}
\bibliography{literature}

\begin{thebibliography}{10}

\bibitem{Sherali2005}
H.~D. Sherali, E.~K. Bish, and X.~Zhu, ``Polyhedral analysis and algorithms for
  a demand-driven refleeting model for aircraft assignment,'' {\em
  Transportation Science}, vol.~39, no.~3, pp.~349--366, 2005.

\bibitem{Jarrah}
A.~I. Jarrah, J.~Goodstein, and R.~Narasimhan, ``An efficient airline
  re-fleeting model for the incremental modification of planned fleet
  assignments,'' {\em Transportation Science}, vol.~34, no.~4, pp.~349--363,
  2000.

\bibitem{Wang}
X.~Wang and A.~Regan, ``Dynamic yield management when aircraft assignments are
  subject to swap,'' {\em Transportation Research Part B}, vol.~40, no.~7,
  pp.~563--576, 2006.

\bibitem{Ageeva}
Y.~Ageeva, ``Approaches to incorporating robustness into airline scheduling,''
  Master's thesis, Department of Aeronautics \& Astronautics, Massachusetts
  Institute of Technology, 2000.

\bibitem{akturk}
M.~S. Akt{\"u}rk, A.~Atamt{\"u}rk, and S.~G{\"u}rel, ``Aircraft rescheduling
  with cruise speed control,'' {\em Operations Research}, vol.~62, no.~4,
  pp.~829--845, 2014.

\bibitem{Arikan}
U.~Ar{\i}kan, S.~G{\"u}rel, and M.~Akt{\"u}rk, ``Flight network-based approach
  for integrated airline recovery with cruise speed control,'' {\em
  Transportation Science}, vol.~51, no.~4, pp.~1259--1287, 2017.

\bibitem{Lonzius}
M.~Lonzius and A.~Lange, ``Robust scheduling: An empirical study of its impact
  on air traffic delays,'' {\em Transportation Research Part E}, vol.~100,
  pp.~98--114, 2017.

\bibitem{virot}
V.~Chiraphadhanakul and C.~Barnhart, ``Robust flight schedules through slack
  re-allocation,'' {\em EURO Journal on Transportation and Logistics}, vol.~2,
  pp.~277--306, 2013.

\bibitem{Dunbar}
M.~Dunbar, G.~Froyland, and W.~Cheng-Lung, ``An integrated scenario-based
  approach for robust aircraft routing, crew pairing and re-timing,'' {\em
  Computers \& Operations Research}, vol.~45, pp.~68--86, 2014.

\bibitem{lans}
S.~Lan, J.~P. Clarke, and C.~Barnhart, ``Planning for robust airline
  operations: Optimizing aircraft routings and flight departure times to
  minimize passenger disruptions,'' {\em Transportation Science}, vol.~40,
  no.~1, pp.~15--28, 2006.

\bibitem{aloulou}
M.~A. Aloulou, M.~Haouari, and F.~Z. Mansour, ``A model for enhancing
  robustness of aircraft and passenger connections,'' {\em Transportation
  Research Part C: Emerging Technologies}, vol.~32, pp.~48--60, 2013.

\bibitem{ahmed}
M.~B. Ahmed, W.~Ghroubi, M.~Haouari, and H.~D. Sherali, ``A hybrid
  optimization-simulation approach for robust weekly aircraft routing and
  retiming,'' {\em Transportation Research Part C: Emerging Technologies},
  vol.~84, pp.~1--20, 2017.

\bibitem{Cadar}
L.~Cadarso and R.~de~Celis, ``Integrated airline planning: Robust update of
  scheduling and fleet balancing under demand uncertainty,'' {\em
  Transportation Research Part C: Emerging Technologies}, vol.~81,
  pp.~227--245, 2017.

\bibitem{Sherali2006}
H.~D. Sherali, R.~W. Staats, and A.~A. Trani, ``An airspace-planning and
  collaborative decision-making model: Part ii-cost model, data considerations,
  and computations,'' {\em Transportation Science}, vol.~40, no.~2,
  pp.~147--164, 2006.

\bibitem{Cook09}
A.~Cook, G.~Tanner, G.~Williams, and G.~Meise, ``Dynamic cost indexing-managing
  airline delay costs,'' {\em Journal of Air Transport Management}, vol.~15,
  no.~1, pp.~26--35, 2009.

\bibitem{FAAEurocontrol}
D.~Knorr, X.~Chen, M.~Rose, J.~Gulding, P.~Enaud, and H.~Hegendoerfer,
  ``Estimating atm efficiency pools in the descent phase of flight,'' in {\em
  9th USA/Europe Air Traffic Management Research and Development Seminar},
  (Berlin, Germany), 2011.

\bibitem{Kang}
L.~Kang and M.~Hansen, ``Improving airline fuel efficiency via fuel burn
  prediction and uncertainty estimation,'' {\em Transportation Research Part C:
  Emerging Technologies}, vol.~97, pp.~128--146, 2018.

\bibitem{Kohl}
N.~Kohl, A.~Larsen, J.~Larsen, A.~Ross, and S.~Tiourine, ``Airline disruption
  management-perspectives, experiences and outlook,'' {\em Journal of Air
  Transport Management}, vol.~13, no.~3, pp.~149--162, 2007.

\bibitem{Marla}
L.~Marla, B.~Vaaben, and C.~Barnhart, ``Integrated disruption management and
  flight planning to trade off delays and fuel burn,'' {\em Transportation
  Science}, vol.~51, no.~1, pp.~88--111, 2016.

\bibitem{Duran}
A.~S. Duran, S.~G{\"u}rel, and M.~S. Akt{\"u}rk, ``Robust airline scheduling
  with controllable cruise times and chance constraints,'' {\em IIE
  Transactions}, vol.~47, no.~1, pp.~64--83, 2015.

\bibitem{Safak}
{\"O}.~\c{S}afak, S.~G{\"u}rel, and M.~S. Akt{\"u}rk, ``Integrated
  aircraft-path assignment and robust schedule design with cruise speed
  control,'' {\em Computers \& Operations Research}, vol.~84, pp.~127--145,
  2017.

\bibitem{Gurkan}
H.~G{\"u}rkan, S.~G{\"u}rel, and M.~S. Akt{\"u}rk, ``An integrated approach for
  airline scheduling, aircraft fleeting and routing with cruise speed
  control,'' {\em Transportation Research Part C: Emerging Technologies},
  vol.~68, pp.~38--57, 2016.

\bibitem{akturk2009strong}
M.~S. Akt{\"u}rk, A.~Atamt{\"u}rk, and S.~G{\"u}rel, ``A strong conic quadratic
  reformulation for machine-job assignment with controllable processing
  times,'' {\em Operations Research Letters}, vol.~37, pp.~187--191, 2009.

\bibitem{AN:cmir}
A.~Atamt{\"u}rk and V.~Narayanan, ``Cuts for conic mixed integer programming,''
  in {\em Proceedings of the 12th International IPCO Conference} (M.~Fischetti
  and D.~P. Williamson, eds.), pp.~16--29, 2007.

\bibitem{europaram}
EUROCONTROL, ``User manual for the base of aircraft data (bada) revision
  3.10.,'' Tech. Rep. 12/04/10-45, EEC Technical/Scientific, Eurocontrol,
  Eurocontrol Experimental Centre, B.P. 15, F-91222 Bretigny-sur-Orge, France,
  2012.

\bibitem{euro}
EUROCONTROL, ``Forecasting civil aviation fuel burn and emissions in europe,''
  Tech. Rep. 2001-8, EEC Technical/Scientific, Eurocontrol, Eurocontrol
  Experimental Centre, B.P. 15, F-91222 Bretigny-sur-Orge, France, 2001.

\bibitem{hoffman}
R.~Hoffman and M.~O. Ball, ``A comparison of formulations for the
  single-airport ground-holding problem with banking constraints,'' {\em
  Operations Research}, vol.~48, no.~4, pp.~578--590, 2000.

\bibitem{delay}
EUROCONTROL, ``European airline delay cost reference values, {V}ersion 4.1.,''
  tech. rep., Transport Studies Group University of Westminster London, 2014.

\bibitem{gunluk}
O.~G{\"u}nl{\"u}k and J.~Linderoth, ``Perspective reformulations of mixed
  integer nonlinear programs with indicator variables,'' {\em Mathematical
  Programming}, vol.~124, no.~1-2, pp.~183--205, 2010.

\bibitem{nemirovski}
A.~Ben-Tal and A.~Nemirovski, {\em Lectures on Modern Convex Optimization:
  Analysis, Algorithms, and Engineering Applications}.
\newblock SIAM, 2001.

\bibitem{cplex}
I.~K. Center, ``What are logical constraints.''
  \url{https://www.ibm.com/support/knowledgecenter/en/SSSA5P_12.3.0/ilog.odms.cplex.help/Content/Optimization/Documentation/Optimization_Studio/_pubskel/ps_usrmancplex1951.html}.
\newblock Visited May 2017.

\bibitem{mannino2018hotspot}
C.~Mannino, A.~Nakkerud, G.~Sartor, and P.~Schittekat, ``Hotspot resolution
  with sliding window capacity constraints using the path\&cycle algorithm,''
  {\em Eighth SESAR Innovation Days}, 2018.

\bibitem{McCormick}
G.~McCormick, ``Computability of global solutions to factorable nonconvex
  programs: Part i-convex underestimating problems,'' {\em Mathematical
  Programming}, vol.~10, no.~1, pp.~147--175, 1976.

\bibitem{BTSperformance}
BTS, ``Airline on-time performance data,'' 2010.
\newblock Visited May 2014.

\bibitem{fuelprice}
IATA, ``Fuel price analysis.''
  \url{http://www.iata.org/publications/economics/fuel-monitor/Pages/price-analysis.aspx},
  2014.
\newblock Visited May 2014.

\bibitem{arikanDes}
M.~Ar{\i}kan, V.~Deshpande, and M.~Sohoni, ``Building reliable air-travel
  infrastructure using empirical data and stochastic models of airline
  networks,'' {\em Operations Research}, vol.~61, no.~1, pp.~45--64, 2012.

\bibitem{eurocrew}
EUROCONTROL, ``Innovative cooperative actions of research \& development in
  eurocontrol programme {CARE} {INO} {III} dynamic cost indexing,'' tech. rep.,
  Transport Studies Group University of Westminster London, 2008.

\end{thebibliography}
\bibliographystyle{ieeetr}

\end{document}